\documentclass[a4paper,11pt]{article}
\usepackage[utf8]{inputenc}
\usepackage[T1]{fontenc}
\usepackage[english]{babel}
\usepackage[margin=2.8cm]{geometry}
\usepackage{mathpazo}
\usepackage[scaled=.95]{helvet}
\usepackage{courier}
\usepackage{amsmath,amsfonts,amssymb,amsthm}
\usepackage{graphicx}
\usepackage{natbib,url}
\usepackage{color}
\usepackage[onehalfspacing]{setspace}
\newcommand{\largtt}[1]{{\large\texttt{#1}}}
\newtheorem{prop}{Proposition}[section]

\newtheorem{rmq}{Remark}[section]  
\newtheorem{theo}{Theorem}[section]  
\newtheorem{lem}{Lemma}[section]

\newtheorem{cor}{Corollary}[section]
\DeclareMathOperator{\dom}{dom}
\DeclareMathOperator{\proj}{proj}
\newcommand*{\nrm}[1]{\left\| #1 \right\|}      
\newcommand*{\E}{\mathbb{E}}
\newcommand*{\EE}[1]{\E\left[#1\right]}
\newcommand{\limite}[2]{\xrightarrow[#1]{#2}}
\newcommand{\CL}{\mathcal{L}}
\DeclareMathOperator{\cvl}{\limite{n \to \infty}{\CL}}

\DeclareMathOperator{\cste}{cste}
\DeclareMathOperator{\var}{Var}
\DeclareMathOperator{\dist}{dist}
\DeclareMathOperator{\eqm}{eqm}
\DeclareMathOperator{\cod}{dom}

\usepackage[draft]{fixme}
\fxusetheme{color}
\FXRegisterAuthor{h}{ah}{H}
\FXRegisterAuthor{p}{ap}{P}
\FXRegisterAuthor{a}{aa}{A}

\begin{document}
\title{Online estimation of the geometric median in Hilbert spaces : non asymptotic confidence balls}
\author{Herv\'e \textsc{Cardot}, Peggy \textsc{C\'enac}, Antoine \textsc{Godichon} \\ Institut de Math\'ematiques de Bourgogne, Universit\'e de Bourgogne, \\
9 Rue Alain Savary, 21078 Dijon, France \\
email: \{herve.cardot, peggy.cenac, antoine.godichon\}@u-bourgogne.fr
} 
\maketitle
\begin{abstract}

Estimation procedures based on recursive algorithms are interesting and powerful techniques that are able to deal rapidly with (very) large samples of high dimensional data. The collected data may be contaminated by noise so that robust location indicators, such as the geometric median, may be preferred to the mean. In this context, an estimator of the geometric median based on a  fast and efficient averaged non linear stochastic gradient algorithm has been developed by \cite{HC}. This work aims at studying more precisely the non asymptotic behavior of this algorithm by giving non asymptotic confidence balls. This new result is based on the derivation of improved $L^2$ rates of convergence as well as an exponential inequality for the martingale terms of the recursive non linear Robbins-Monro algorithm. 
     
\end{abstract}

\section{Introduction}

Dealing with large samples of observations taking values in high dimensional spaces such as functional spaces is not unusual nowadays. In this context, simple estimators of location such as the arithmetic mean can be greatly influenced by a small number of outlying values.  Thus, robust indicators of location may be preferred to the mean.
We focus in this work on the estimation of the geometric median, also called $L^{1}$-median or spatial median. It is a generalization of the real median introduced by \cite{Hal48}  that can now be computed rapidly, even for large samples in high dimension spaces, thanks to recursive algorithms (see \cite{HC}). 

Let $H$ be a separable Hilbert space, we denote by $\langle . , . \rangle$ its inner product and by $\| . \|$ the associated norm. Let $X$ be a random variable taking values in $H$, the geometric median $m$ of $X$ is defined by:
\begin{equation}
\label{defi}m:= \arg \min_{h\in H} \mathbb{E}\left[ \| X-h \| - \| X \| \right].
\end{equation}

Many properties of this  median in separable Banach spaces are given by \cite{Kem87} such as  existence and uniqueness, as well as robustness (see also the review   \cite{Sma90}).  Recently, this median has received much attention in the literature.  For example, \cite{minsker2013geometric} suggests to consider, in various statistical contexts, the geometric  median of independent estimators to obtain much tighter concentration bounds. In functional data analysis, \cite{KrausPanaretos2012} consider resistant estimators of the covariance operators based on the geometric median in order to derive a robust test of  equality of the second-order structure for two samples. The geometric median is also chosen to be the central location indicator in various types of robust functional principal components analyses (see   \cite{LMSTZC1999}, \cite{Ger08} and \cite{BBTW2011}).  Finally, a general definition of the geometric median on manifolds is given in \cite{ADPY10} with signal processing issues in mind.

Consider a sequence of i.i.d copies $X_1, X_2, \ldots, X_n, \ldots $ of $X$. A natural estimator $\widehat{m}_n$ of $m$, based on $X_1, \ldots, X_n$,  is obtained by  minimizing the empirical risk
\begin{equation}
\label{defi-emp}
\widehat{m}_n := \arg \min_{h\in H} \sum_{i=1}^n \left[ \| X_i-h \| - \| X_i \| \right].
\end{equation}
Convergence properties of  the empirical estimator $\widehat{m}_n$ are reviewed in \cite{MNO2010} when the dimension of $H$ is finite whereas the recent work of \cite{chakraborty2014spatial} proposes a deep asymptotic study for random variables taking values in separable Banach spaces. 

Given a sample $X_1, \ldots, X_n$,  the computation of $\widehat{m}_n$ generally relies on a variant of the Weiszfeld's algorithm  (see {\it e.g.} \cite{kuhn1973note}) introduced by  \cite{VZ00}. This iterative algorithm is relatively fast  (see \cite{beck2013weiszfeld} for an improved version) but it is not adapted to handle very large data sets of high-dimensional data since it requires to store all the data. However huge datasets are not unusual anymore with the development of automatic sensors and smart meters. In this context  \cite{HC} have developed a much faster recursive algorithm, which does not require to store all the data and can be updated automatically when the data arrive online. The estimation procedure is based on the simple following recursive scheme, 
\begin{equation}\label{algo}
Z_{n+1} = Z_{n}+\gamma_{n}\frac{X_{n+1}-Z_{n}}{\nrm{X_{n+1}-Z_{n}}} 
\end{equation}
where the  sequence of steps $\left( \gamma_{n} \right)$ controls the convergence of the algorithm and satisfy the usual conditions for the convergence of Robbins Monro algorithms (see Section~\ref{sectionrateofconvergence}). The averaged version of the algorithm is given by
\begin{align}
\overline{Z}_{n+1} &=\overline{Z}_{n}+\frac{1}{n+1}\left( Z_{n+1}-\overline{Z}_{n}\right),
\label{algo:moy}
\end{align}
with $\overline{Z}_{0}=0$, so that $\overline{Z}_{n}=\frac{1}{n}\sum_{i=1}^{n}Z_{i}$.  
The averaging step described in (\ref{algo:moy}), and first studied in \cite{PolyakJud92}, allows a considerable improvement of  the convergence of the initial Robbins-Monro algorithm. It is shown in \cite{HC} that the recursive averaged estimator $\overline{Z}_n$ and the empirical estimator $\widehat{m}_n$   have the same Gaussian limiting distribution. In infinite dimensional spaces, this nice result heavily relies on the (locally) strong convex properties of the objective function to be minimized. Note that  \cite{bach2014adaptivity} adopts  an analogous recursive point of view for logistic regression under slightly different conditions, called self-concordance, which involve uniform conditions on the third order derivatives of the objective function.

\medskip

The aim of this work is to give new arguments in favor of the averaged stochastic gradient algorithm 
by providing a sharp control of its deviations around the true median, for finite samples. To get such non asymptotic confidence balls, new results about the behavior of the stochastic algorithm are proved : improved convergence rates in quadratic mean compared to those obtained in \cite{HC}  as well as new exponential inequalities for "near" martingale sequences in Hilbert spaces, similar to the seminal result of \cite{Pinelis} for martingales. Note that, as far as we  know, there are only very few results in the  literature on exponential bounds for non linear recursive algorithms (see however \cite{balsubramani2013fast} for recursive PCA).

\medskip

The paper is organized as follows. Section~\ref{sectiondefi} recalls some convexity properties of the geometric median as well as the basic assumptions ensuring the uniqueness of the geometric median. 
In Section~\ref{sectionrateofconvergence}, the rates of convergence of the stochastic gradient algorithm are derived in quadratic mean as well as in $L^4$. 
In Section~\ref{sectioninterv},   an exponential inequality is derived borrowing ideas from \cite{Tarres}. It enables us to build non asymptotic confidence balls for the Robbins-Monro algorithm as well as  its averaged version. All the proofs are gathered in Section~\ref{sectionproof}.

\section{Assumptions on the median and convexity properties}\label{sectiondefi}
Let us first state basic assumptions on the median.
 \begin{itemize}
\item[\textbf{(A1)}] The random variable $X$ is not concentrated on a straight line: for all $h\in H$, there exists $h' \in H$ such that $\langle h,h' \rangle =0$ and
\[
\var \left( \langle h',X \rangle \right) >0.
\]
\item[\textbf{(A2)}] $X$ is not concentrated around single points: there is a constant $C>0$ such that for all $h \in H$:
\[
\mathbb{E}\left[ \| X-h \|^{-1} \right] \leq C. 
\] 
\end{itemize}
Assumption \textbf{(A1)} ensures that the median $m$ is uniquely defined \citep{Kem87}. Assumption \textbf{(A2)} is closely related to small ball probabilities and to the dimension of $H$. It was proved in \cite{Cha92} that  when $H= \mathbb{R}^{d}$, assumption \textbf{(A2)} is satisfied when  $d \geq 2$ under classical assumptions on the density of $X$. A detailed discussion on assumption \textbf{(A2)} and its connection with small balls probabilities can be found in \cite{HC}.

We now recall some results about convexity and robustness of the geometric median. We denote by $G:H \longrightarrow \mathbb{R}$ the convex function we would like to minimize, defined for all $h\in H$ by
\begin{equation}
\label{defiG}G(h):= \mathbb{E} \left[ \| X-h \| -\| X\| \right]. 
\end{equation}
This function is Fr\'echet differentiable on $H$, we denote by $\Phi$ its Fr\'echet derivative, and for all $h \in H$:
\[\Phi (h):= \nabla_{h}G=-\mathbb{E}\left[ \frac{X-h}{\| X-h \|} \right].\]
Under previous assumptions,  $m$ is the unique zero of $\Phi$.

Let us define $U_{n+1}:= -\frac{X_{n+1}-Z_{n}}{\| X_{n+1}-Z_{n} \|}$ and let us introduce the sequence of $\sigma$-algebra $\mathcal{F}_{n}:=\sigma \left( Z_{1},...,Z_{n} \right) = \sigma \left(X_{1},...,X_{n}\right)$. For all integer $n \geq 1$,
\begin{align}
\mathbb{E}\left[ U_{n+1}| \mathcal{F}_{n} \right] &= \Phi (Z_{n}).
\end{align}
The sequence $\left( \xi_{n} \right)_n$ defined by $\xi_{n+1}:=\Phi(Z_{n})-U_{n+1}$ is a martingale difference sequence with respect to the filtration $\left( \mathcal{F}_{n} \right)$. Moreover, we have for all $n$, $\| \xi_{n+1} \| \leq 2$ and
\begin{align}
\label{majxi1} 
\mathbb{E}\left[ \| \xi_{n+1}\|^{2} |\mathcal{F}_{n} \right] &\leq 1- \| \Phi (Z_{n}) \|^{2} \leq 1 .
\end{align}
Algorithm (\ref{algo}) can be written as a Robbins-Monro or  a stochastic gradient algorithm:
\begin{align}
\label{decphi} Z_{n+1}-m &= Z_{n}-m-\gamma_{n}\Phi (Z_{n})+\gamma_{n}\xi_{n+1}.
\end{align}

We now consider the Hessian of $G$,
 which is denoted  by $\Gamma_{h}:H \longrightarrow H$. It satisfies (see \cite{Ger08})
\[\Gamma_{h}=\mathbb{E}\left[ \frac{1}{\| X-h \|}\left( I_{H}-\frac{(X-h)\otimes (X-h)}{\| X-h \|^{2}}\right) \right],\]
where $I_{H}$ is the identity operator in $H$ and $u \otimes v(h)=\langle u,h \rangle v$ for all $u,v,h \in H$. 
The following (local) strong convexity properties will be useful (see  \cite{HC} for proofs).
\begin{prop}[ \cite{HC} ]\label{propgamma}
Under assumptions \textbf{(A1)} and \textbf{(A2)}, for any real number $A>0$, there is a  positive constant $c_{A}$ such that for all $h\in H$ with $\| h \| \leq A$, and for all $h'\in H$:
\[
 c_{A}\| h' \|^{2} \leq \langle h' , \Gamma_{h}h' \rangle \leq C \| h' \|^{2} .
\]
As a particular case, there is a positive constant $c_{m}$ such that for all $h'\in H$:
\begin{equation}
\label{encgamma} c_{m} \| h' \|^{2} \leq \langle h',\Gamma_{m}h' \rangle \leq C \| h' \|^{2}.
\end{equation}
\end{prop}
The following corollary recall some properties of the spectrum of the Hessian of $G$, in particular on the spectrum of $\Gamma_{m}$.
\begin{cor}\label{corgamma}
Under assumptions \textbf{(A1)} and \textbf{(A2)}, for all $h\in H$, there is an increasing sequence of non-negative eigenvalues $\left( \lambda_{j,h}\right)$ and an orthonormal basis $\left( v_{j,h}\right)$ of eigenvectors of $\Gamma_{h}$ such that
\begin{align*}
\Gamma_{h}v_{j,h} &  = \lambda_{j,h}v_{j,h}, \\
\sigma (\Gamma_{h} ) & =\left\lbrace \lambda_{j,h},j\in \mathbb{N}\right\rbrace , \\
\lambda_{j,h} & \leq C .
\end{align*}
Moreover, if $\| h \| \leq A$, for all $j \in \mathbb{N}$ we have $c_{A} \leq \lambda_{j,h} \leq C$. 

As a particular case, the  eigenvalues $ \lambda_{j,m} $  of $\Gamma_{m}$ satisfy,  $c_{m} \leq \lambda_{j,m} \leq C$, for all $j \in \mathbb{N}$.
\end{cor}
The bounds are an immediate consequence of  Proposition \ref{propgamma}. Remark that with these different convexity properties of the geometric median, we are close to the  framework of \cite{bach2014adaptivity}. The difference comes from the fact that  $G$ does not satisfy the generalized self-concordance assumption which is central in the latter work. 



\section{Rates of convergence of the Robbins-Monro algorithms}\label{sectionrateofconvergence}



If the  sequence $\left( \gamma_{n} \right)_n$  of stepsizes fulfills the classical following assumptions:
\[
\sum_{n\geq 1}\gamma_{n}^{2}<\infty    \quad \mbox{ and } \quad  \sum_{n\geq 1}\gamma_{n} = \infty,
\]
 the recursive estimator  $Z_n$ is strongly consistent (see  \cite{HC}).
The first condition on the stepsizes ensures that the recursive algorithm converges towards some value in $H$ whereas the second condition forces the algorithm to converge to $m$, the unique minimizer of $G$.

From now on, $Z_{1}$ is chosen so that it is bounded (consider for example  $Z_{1}=X_{1}\mathbb{1}_{\{\| X \| \leq M'\}}$ for some non negative constant $M'$).
Consequently, there is a positive constant $M$ such that for all $n \geq 1$:
\[
 \mathbb{E}\left[ \| Z_{n}-m \|^{2} \right] \leq M.
\]

Let us consider now sequences $\left( \gamma_{n}\right)_n$ of the form $\gamma_{n}=c_{\gamma}n^{-\alpha}$ where $c_{\gamma}$ is a positive constant, and $\alpha \in (1/2,1)$. In order to get  confidence balls for the median, the following additional  assumption is supposed to hold.

\begin{itemize}
\item[$\textbf{(A3)}$] There is a positive constant $C$ such that for all $h \in H$:
\[
 \mathbb{E}\left[ \| X-h \|^{-2}\right] \leq C.
\]  
\end{itemize}
This assumption ensures that  the remainder  term in the Taylor  approximation to  the gradient is bounded. Note that this assumption is also required to get the asymptotic normality in \cite{HC}.
 It is also assumed  in \cite{chakraborty2014spatial} for deriving the asymptotic normality of the empirical median estimator.  Remark that for the sake of simplicity, we have considered  the same constant $C$ in \textbf{(A2)} and \textbf{(A3)}. As in \textbf{(A2)}, Assumption \textbf{(A3)} is closely related to small ball probabilities and  when $H = \mathbb{R}^{d}$, this assumption is satisfied when $d\geq 3$  under weak conditions.

We state now the first new and important result on the rates of convergence in quadratic mean of the Robbins Monro algorithm. 
A comparison with Proposition~3.2 in \cite{HC} reveals that the logarithmic term has disappeared as well as the constant $C_N$ that was related to a sequence $(\Omega_N)_N$ of events whose probability was tending to one.
 \begin{theo}\label{bonnevitesse}
Assuming  \textbf{(A1)-(A3)} hold, the algorithm $\left( Z_{n}\right)$ defined by (\ref{algo}), with $\gamma_{n}=c_{\gamma}n^{-\alpha}$, converges in quadratic mean, for all $\alpha \in (1/2,1)$ and for all $\alpha <\beta<3\alpha -1$, with the following rate:
\begin{align}
\mathbb{E}\left[ \| Z_{n}-m \|^{2} \right] & = O\left( \frac{1}{n^{\alpha}}\right) ,\\
\mathbb{E}\left[ \| Z_{n}-m  \|^{4}\right] & = O \left( \frac{1}{n^{\beta}}\right) .
\end{align}
\end{theo}
Upper bounds for the rates of convergence at  order four are also given because they will be useful in several proofs. Remark that obtaining  better rates of convergence at the order four would also be possible at the expense of longer proofs, but it is not necessary here. The proof of this theorem relies on two technical lemmas. The following one gives an upper bound of the quadratic mean error. 
\begin{lem}\label{majexp}
Assuming  \textbf{(A1)-(A3)} hold, there are positive constants $C_{1},C_{2},C_{3},C_{4}$ such that for all $n\geq 1$:
\begin{equation}
\mathbb{E}\left[ \| Z_{n}-m \|^{2}\right] \leq C_{1}e^{-C_{4}n^{1-\alpha}}+\frac{C_{2}}{n^{\alpha}}+C_{3}\sup_{n/2-1 \leq k \leq n}\mathbb{E}\left[ \| Z_{k}-m \|^{4}\right] .
\end{equation}
\end{lem}
The proof of Lemma~\ref{majexp} is given in Section~\ref{sectionproof}.

\begin{lem}\label{propor4}
Assuming the three assumptions \textbf{(A1)} to \textbf{(A3)}, for all $\alpha \in (1/2,1)$, there are a rank $n_{\alpha}$ and positive constants $C_{1}',C_{2}'$ such that for all $n \geq n_{\alpha}$:
\begin{align}\label{maj4}
\mathbb{E}\left[ \| Z_{n+1}-m \|^{4}\right] & \leq \left( 1-\frac{1}{n}\right)^{2}\mathbb{E}\left[ \| Z_{n}-m \|^{4}\right]+\frac{C_{1}'}{n^{3\alpha}}+C_{2}'\frac{1}{n^{2\alpha}} \mathbb{E}\left[ \| Z_{n}-m \|^{2}\right] .
\end{align}
\end{lem}
The proof of Lemma \ref{propor4} is given in Section~\ref{sectionproof}.
The next result gives the exact rate of convergence in quadratic mean and states that it is not possible to get the parametric rates of convergence with the Robbins Monro algorithm when $\alpha \in (1/2,1)$.
\begin{prop}\label{propbonnevitesse}
Assume \textbf{(A1)-(A3)} hold, for all $\alpha \in (1/2,1)$, there is a positive constant $C'$ such that for all $n \geq 1$,
\begin{align*}
\mathbb{E}\left[ \left\| Z_{n}-m \right\|^{2} \right] &\geq \frac{C'}{n^{\alpha}}.
\end{align*}
\end{prop}

\section{Non asymptotic confidence balls}\label{sectioninterv}
\subsection{Non asymptotic confidence balls for the Robbins-Monro algorithm}\label{sectionpinelis}

The aim is now to derive an upper bound for  $\mathbb{P}\left[ \| Z_{n}-m \| \geq t \right]$, for $t >0$. 
A simple first result can be obtained by applying Markov's inequality and  Theorem \ref{bonnevitesse}. We give below a sharper bound that relies on exponential inequalities that are close to the ones given in Theorem 3.1 in \cite{Pinelis}. As explained in Remark~\ref{rmq:Pinelis} below, it was not possible to apply directly Theorem 3.1 of \cite{Pinelis} and the following proposition gives an analogous exponential inequality in the case where we do not have exactly a sequence of martingale differences. 
\begin{prop}\label{pinelisadapte}
Let $\left( \beta_{n,k} \right)_{(k,n)\in \mathbb{N}\times \mathbb{N}}$ be a sequence of linear operators on $H$ and $\left( \xi_{n} \right)$ be a sequence of $H$-valued martingale differences adapted to a filtration $\left( \mathcal{F}_{n} \right)$. Moreover, let $\left( \gamma_{n} \right)$ be a sequence of positive real numbers. Then, for all $r >0$ and for all  $n \geq 1$,
\begin{align*}
\mathbb{P}\left[ \left\| \sum_{k=1}^{n-1}\gamma_{k}\beta_{n-1,k}\xi_{k+1} \right\| \geq r \right] & \leq 2e^{-r}\left\| \prod_{j=2}^{n}\left( 1+\mathbb{E}\left[ e^{\left\| \beta_{n-1,j-1}\gamma_{j-1}\xi_{j}\right\|}-1-\left\| \beta_{n-1,j-1}\gamma_{j-1}\xi_{j}\right\| \Big| |\mathcal{F}_{j-1}\right]\right) \right\| \\
&  \leq 2\exp \left( -r+\left\| \sum_{j=2}^{n}\mathbb{E}\left[ e^{\left\| \beta_{n-1,j-1}\gamma_{j-1}\xi_{j}\right\|}-1-\left\| \beta_{n-1,j-1}\gamma_{j-1}\xi_{j}\right\| \Big| |\mathcal{F}_{j-1}\right] \right\| \right) .
\end{align*}
\end{prop}
The proof of Proposition \ref{pinelisadapte} is postponed to Section \ref{sectionproof}. 
As in \cite{Tarres}, it enables to give a sharp upper bound for $\mathbb{P}\left[ \left\| \sum_{k=1}^{n-1}\beta_{n-1,k}\gamma_{k}\xi_{k+1} \right\| \geq t \right]$.

\begin{cor}\label{corpinbern1}
Let $\left( \beta_{n,k} \right)$ be sequence of linear operators on $H$, $\left( \xi_{n} \right)$ be a sequence of $H$-valued martingale differences adapted to a filtration $\left( \mathcal{F}_{n}\right)$ and $\left( \gamma_{n} \right)$ be a sequence of positive real numbers. Let $\left( N_{n} \right)$ and $\left(\sigma_{n}^{2}\right)$ be two deterministic sequences such that
\[N_{n} \geq \sup_{k\leq n-1}\| \beta_{n-1,k}\gamma_{k}\xi_{k+1}\| \quad a.s. \quad \mbox{and} \quad \sigma_{n}^{2} \geq \sum_{k=1}^{n-1}\mathbb{E}\left[ \| \beta_{n-1,k}\gamma_{k}\xi_{k+1}\| \big| \mathcal{F}_{n} \right].\]
 For all $t > 0$ and all $n \geq 1$, 
\begin{align*}
\mathbb{P}\left[ \left\| \sum_{k=1}^{n-1}\beta_{n-1,k}\gamma_{k}\xi_{k+1} \right\| \geq t \right] & \leq 2\exp \left( -\frac{t^{2}}{2(\sigma_{n}^{2} + tN_{n}/3)}\right).
\end{align*} 
\end{cor}
In order to apply these results, let us linearize the gradient around $m$ in decomposition~(\ref{decphi}),
\begin{align}\label{decdelta}
 Z_{n+1}-m & = Z_{n}-m -\gamma_{n}\Gamma_{m}( Z_{n}-m)+\gamma_{n} \xi_{n+1} -\gamma_{n}\delta_{n},
\end{align}
where $\delta_{n}:=\Phi(Z_{n})-\Gamma_{m}(Z_{n}-m)$ and introduce, for all $n\geq 1$, the following operators:
\begin{align*}
\alpha_{n} & := I_{H}-\gamma_{n}\Gamma_{m}, \\
\beta_{n} & := \prod_{k=1}^{n}\alpha_{k} = \prod_{k=1}^{n}\left(I_{H}-\gamma_{k}\Gamma_{k}\right) , \\
\beta_{0} & := I_{H}.
\end{align*}
By induction, (\ref{decdelta}) yields
\begin{align}
\label{decbeta} Z_{n}-m & = \beta_{n-1}(Z_{1}-m)+\beta_{n-1}M_{n}-\beta_{n-1}R_{n},
\end{align}
with
\begin{align*}
R_{n} & := \sum_{k=1}^{n-1}\gamma_{k}\beta_{k}^{-1}\delta_{k}, \\
M_{n} & := \sum_{k=1}^{n-1}\gamma_{k}\beta_{k}^{-1}\xi_{k+1}.
\end{align*}
\begin{rmq}\label{rmq1}
Note that we make an abuse of notation because $\beta_{k}^{-1}$ does not necessarily exist. However,  if $c_{\gamma} < \frac{1}{C}$,  the linear operator $\beta_{k}^{-1}$ is bounded. Moreover, we can make this abuse because, even if $\beta_{k}$ has not a continuous inverse, we only need to consider $\beta_{n-1}\beta_{k}^{-1}:=\prod_{j=k+1}^{n-1}\left( I_{H} - \gamma_{j}\Gamma_{m} \right)$, which are continuous operators for  $k \leq n-1$.
\end{rmq}
Note that, if $\beta_{k}$ is invertible for all $k\geq 1$, $\left( M_{n} \right)$ is a martingale sequence adapted to the filtration $\left( \mathcal{F}_{n} \right)$. Moreover,
\begin{align}\label{majopznmt}
\notag \mathbb{P}\left[ \| Z_{n}-m \| \geq t \right] & \leq \mathbb{P}\left[ \| \beta_{n-1}M_{n} \| \geq \frac{t}{2} \right] + \mathbb{P}\left[ \| \beta_{n-1}R_{n} \| \geq \frac{t}{4} \right] + \mathbb{P}\left[ \| \beta_{n-1}(Z_{1}-m)\| \geq \frac{t}{4}\right] \\
& \leq \mathbb{P}\left[ \| \beta_{n-1}M_{n} \| \geq \frac{t}{2}\right] +4\frac{\mathbb{E}\left[ \| \beta_{n-1}R_{n}\| \right]}{t}+ 16\frac{\mathbb{E}\left[ \| \beta_{n-1}(Z_{1}-m))\|^{2}\right]}{t^{2}}.
\end{align} 
In this context, Corollary \ref{corpinbern1} can be written as follows:
\begin{cor}\label{corpinbern}
Let $\left( N_{n} \right)_{n \geq 1}$ and $\left( \sigma_{n}^{2} \right)_{n\geq 1}$ be two deterministic sequences such that
\[
N_{n} \geq \sup_{k \leq n-1}\left\| \beta_{n-1}\beta_{k}^{-1}\gamma_{k}\xi_{k+1} \right\| \quad a.s. \quad \text{and} \quad \sigma_{n}^{2} \geq \sum_{k=1}^{n-1}\mathbb{E}\left[ \left\| \beta_{n-1}\beta_{k}^{-1}\gamma_{k}\xi_{k+1} \right\| \big| \mathcal{F}_{n} \right] .
\]
Then, for all $t > 0$ and for all $n \geq 1$,
\[
\mathbb{P}\left[ \left\| \sum_{k=1}^{n-1}\beta_{n-1}\beta_{k}^{-1}\gamma_{k}\xi_{k+1}\right\| \geq t \right] \leq 2 \exp \left( - \frac{t^{2}}{2( \sigma_{n}^{2} + tN_{n} /3 )} \right) .
\]
\end{cor}


We can now derive non asymptotic confidence balls for the Robbins Monro algorithm.
\begin{theo}\label{interv}
Assume that \textbf{(A1)-(A3)} hold. There is a positive constant $C$ such that for all $\delta \in (0,1)$, there is a rank $n_{\delta}$ such that for all $n\geq n_{\delta}$,
\begin{align*}
\mathbb{P} \left[ \| Z_{n}-m \|  \leq  \frac{C}{n^{\alpha /2}}\ln \left( \frac{4}{\delta}\right) \right] &\geq 1 - \delta .
\end{align*} 
\end{theo}
\begin{rmq}\label{rmq:Pinelis}
Note that we could not apply Theorem 3.1 in \cite{Pinelis} to the martingale term $M_{n}= \sum_{k=1}^{n-1}\beta_{k}^{-1}\gamma_{k}\xi_{k+1}$. In fact, two problems are encountered. First, as written in Remark \ref{rmq1}, $\beta_{k}^{-1}$ does not necessarily exist. The second problem is that although there is a positive constant $M$ such that $\left\| \beta_{n-1}M_{n} \right\| \leq M$ for all $n \geq 1$, the sequence $\left\| \beta_{n-1} \right\| \left\|  M_{n} \right\|$ may not be convergent ($\left\| \beta_{n-1} \right\|$ denotes the usual spectral norm of operator $\beta_{n-1}$).
\end{rmq}
\subsection{Non asymptotic confidence balls for the averaged algorithm:}
As  in \cite{HC} and \cite{Pel00}, we make use of decomposition (\ref{decdelta}).
By summing 
and applying Abel's transform, we get:
\begin{equation}\label{decinterv}
\Gamma_{m}\overline{T}_{n} = \frac{1}{n}\left( \frac{T_{1}}{\gamma_{1}}-\frac{T_{n+1}}{\gamma_{n}}+ \sum_{k=2}^{n} T_{k}\left[ \frac{1}{\gamma_{k}}-\frac{1}{\gamma_{k+1}} \right] - \sum_{k=1}^{n}\delta_{k} \right) + \frac{1}{n}\widehat{M}_{n+1} ,
\end{equation}
with
\begin{align*}
T_{n} & := Z_{n}-m , \\
\overline{T}_{n} & := \overline{Z}_{n} -m \\
\widehat{M}_{n+1} & := \sum_{k=1}^{n}\xi_{k+1} .
\end{align*}
The last term is the martingale term. Applying Pinelis-Bernstein's Lemma (see \cite{Tarres}, Appendix A) to this term and showing that the other ones are negligible, we get the following non asymptotic confidence balls.
\begin{theo}\label{theointervmoyennise}
Assume that \textbf{(A1)-(A3)} hold. For all $\delta \in (0,1)$, there is a rank $n_{\delta}$ such that for all $n \geq n_{\delta}$,
\begin{align*}
\mathbb{P}\left[ \left\| \Gamma_{m} (\overline{Z}_{n}-m )  \right\|  \leq 4 \left( \frac{2}{3n}+\frac{1}{\sqrt{n}}\right) \ln \left(\frac{4}{\delta}\right)\right] &\geq 1-\delta.
\end{align*}
\end{theo}
Since the smallest eigenvalue $\lambda_{\min}$ of $\Gamma_{m}$ is strictly positive, 
\begin{align*}
\mathbb{P} \left[  \left\| \overline{Z}_{n}-m \right\|  \leq \frac{4}{\lambda_{\min}} \left( \frac{2}{3n}+\frac{1}{\sqrt{n}}\right) \ln \left( \frac{4}{\delta} \right) \right] &\geq 1-\delta .
\end{align*}

\begin{rmq} We can also have a more precise form of the rank $n_\delta$ (see the Proof of Theorem~\ref{theointervmoyennise}):
\begin{align}
n_{\delta} &:= \max \left\lbrace \left(\frac{6C_{1}'}{\delta \ln \left( \frac{4}{\delta}\right)} \right)^{\frac{1}{1/2 - \alpha /2}} , \left(\frac{6C_{2}'}{\delta \ln \left( \frac{4}{\delta}\right)}\right)^{\frac{1}{\alpha -1/2}}  , \left(\frac{6C_{3}'}{\delta \ln \left( \frac{4}{\delta}\right)}\right)^{\frac{1}{2}}  \right\rbrace,
\label{def:ndelta}
\end{align}
where $C_{1}', C_{2}'$ and $C_{3}'$ are constants. We can remark that the first two terms are the leading ones and if the rate $\alpha$ is chosen equal to $2/3$, they are of the same order that is $n_{\delta} = O \left(\frac{-1}{\delta \ln \delta} \right)^6$.
\end{rmq}
\begin{rmq} We can make an informal comparison of previous result with the central limit theorem stated in  (\cite{HC}, Theorem 3.4),  even if the latter result is of asymptotic nature. Under assumptions \textbf{(A1)-(A3)}, they have shown that
\[
\sqrt{n} \left( \overline{Z}_n - m \right)  \cvl
  \mathcal{N} \left(0, \Gamma_m^{-1}\Sigma \Gamma_m^{-1}\right),
\]
with, 
\[
\Sigma =  \EE{\frac{(X-m)}{\nrm{X-m}} \otimes \frac{(X-m)}{\nrm{X-m}}}.
\]
This implies, with the continuity of the norm in $H$, that for all $t>0$,
\[
\lim_{n\to \infty}\mathbb{P}\left[ \nrm{ \sqrt{n} \left( \overline{Z}_n - m \right)} \geq t \right]=\mathbb{P}\left[ \nrm{ V } \geq t \right],
\]
where $V$ is a centered $H$-valued Gaussian random vector with covariance operator $\Delta_V = \Gamma_m^{-1}\Sigma \Gamma_m^{-1}$. Operator $\Delta_V$ is self-adjoint and non negative, so that it admits a spectral decomposition $\Delta_V = \sum_{j \geq 1} \eta_j v_j \otimes v_j$,
where $\eta_1 \geq \eta_2 \geq .... \geq 0$ is the sequence of ordered eigenvalues associated to the orthonormal eigenvectors $v_1, v_2, \ldots$ 
Using the Karhunen-Lo\`eve expansion of $V$, we directly get that 
\begin{align*}
\nrm{V}^2 & = \sum_{j \geq 1} \eta_j^2 V_j^2 
\end{align*}
where $V_1, V_2, \ldots$ are {\it i.i.d.} centered Gaussian variables with unit variance. Thus the distribution of $\nrm{V}^2$ is a mixture of independent Chi-square random variables with one degree of freedom. Computing the quantiles of $\nrm{V}$ to build confidence balls would require to know, or to estimate,  all the (leading) eigenvalues of the rather complicated operator $\Delta_V$ and this is not such an easy task. 

On the other hand, the use of the confidence balls given in Theorem~\ref{theointervmoyennise} only requires the knowledge of $\lambda_{\min}$. This eigenvalue is not difficult to estimate since it can also be written as
\[
\lambda_{\min} = \EE{\frac{1}{\nrm{X-m}}} - \lambda_{\max} \left(\EE{\frac{1}{\nrm{X-m}^3} (X-m)\otimes (X-m)}\right),
\]
where $\lambda_{\max}(A)$ denotes the largest eigenvalue of operator $A$. 
\end{rmq}
\begin{rmq} Under previous assumptions and the additional condition $\alpha>2/3$, 
 it can be shown with  decomposition (\ref{decinterv}) that there is a positive constant $C'$ such that
\begin{align*}
 \mathbb{E}\left[ \| \overline{Z}_{n}-m \|^{2} \right] &\leq \frac{C'}{n}.
\end{align*}
The averaged algorithm converges at the parametric rate of convergence in quadratic mean. 
\end{rmq}

\section{Proofs}\label{sectionproof}
\subsection{Proofs of the results given in Section \ref{sectionrateofconvergence}}
In order to prove Lemma \ref{majexp}, we have to introduce a technical lemma which controls the remainder term  $\| \delta_{n} \|$ (see eq. \ref{decdelta}) appearing in the Taylor approximation. This will enable us to bound the term $\beta_{n-1}R_{n}$ in decomposition (\ref{decbeta}).
\begin{lem}\label{lemdelta}
Assuming assumption \textbf{(A3)}, there is a constant $C_{m}$ such that for all $n\geq 1$, almost surely:
\begin{equation}
\label{majdelta} \| \delta_{n} \| \leq C_{m}\| Z_{n}-m\|^{2},
\end{equation}
where $\delta_{n} := \Phi (Z_{n}) - \Gamma_{m}(Z_{n}-m)$.
\end{lem}
\begin{proof}[ Proof of Lemma \ref{lemdelta}]
Using Taylor's theorem with remainder of integral form, almost surely
\begin{equation}
\Phi (Z_{n})=\int_{0}^{1}\Gamma_{m+t(Z_{n}-m)}(Z_{n}-m)dt,
\end{equation}
and 
\[\delta_{n}=\int_{0}^{1}\left( \Gamma_{m+t(Z_{n}-m)}-\Gamma_{m}\right)(Z_{n}-m)dt.\] 
For all $h,h' \in H$, we denote by $\varphi_{h,h'}$ the function defined as follows:
\begin{align*}
\varphi_{h,h'}:[0,1] & \longrightarrow H \\
t & \longmapsto \varphi_{h,h'}(t):=\Gamma_{m+th}(h').
\end{align*}
Let $U_{h}:[0,1]\longrightarrow \mathbb{R}_{+}$ and $V_{h,h'}:[0,1]\longrightarrow H$  be two random functions defined for all $t\in [0,1]$ by
\begin{align*}
U_{h}(t) & := \frac{1}{\| X-m-th \|}, \\
V_{h,h'}(t) & := h'-\frac{\langle X-m-th,h' \rangle (X-m-th)}{\| X-m-th\|^{2}}.
\end{align*} 
Let $V_{h,h'}'(t)= \frac{d}{dt}V_{h,h'}(t)=\lim_{t'\rightarrow 0} \frac{v_{h,h'}(t+t')-v_{h,h'}(t)}{t'}$ and $U_{h,}'(t)= \frac{d}{dt}U_{h}(t) = \lim_{t' \rightarrow 0} \frac{u_{h,h'}(t+t') - u_{h,h'}(t)}{t}$. Let $\varphi_{h,h'}'(t) = \frac{d}{dt}\varphi_{h,h'}(t)$, by dominated convergence, $\varphi_{h,h'}$ is differentiable on $[0,1]$ and
$ \| \varphi_{h,h'}'(t)\| \leq \mathbb{E}\left[ | U_{h}'(t) | \| V_{h,h'}(t) \| + | U_{h}(t) | \| V_{h,h'}'(t) \| \right]$. Using Cauchy-Schwarz inequality, 
\begin{align*}
\left| U_{h}(t)\right| & = \frac{1}{\| X-m-th \|},\\
|U_{h}'(t)| & \leq \frac{\| h \|}{\| X-m-th\|^{2}},\\
\| V_{h,h'}(t) \| & \leq 2 \| h' \| ,\\
\| V_{h,h'}'(t) \| & \leq \frac{4\| h \| \| h' \|}{\| X-m-th \|}.
\end{align*}   
Finally, using assumption \textbf{(A3)}, 
\begin{align*}
\| \varphi_{h,h'}(t) \| & \leq 6 \| h \| \| h' \| \mathbb{E}\left[ \frac{1}{\| X-m-th \|^{2}}\right] \\
& \leq 6 \| h \| \| h' \| C.
\end{align*}
Using previous inequalities, we obtain that for all $h\in H$
\begin{align*}
\| \Phi(m+h) -\Gamma_{m}(h)\| & \leq \int_{0}^{1}\left\| \Gamma_{m+th}(h)-\Gamma_{m}(h) \right\|dt \\
& \leq \int_{0}^{1}\| \varphi_{h,h}(t)-\varphi_{h,h}(0)\| dt \\
& \leq \int_{0}^{1}\sup_{t'\in [0,t]}\| \varphi_{h,h}'(t')\| dt \\
& \leq 6C \| h \|^{2}.
\end{align*}
Taking $h=Z_{n}-m$, for all $n \geq 1$:
\[
 \| \delta_{n} \| \leq C_{m} \| Z_{n}-m \|^{2},
\]
with $C_{m}=6C$.
\end{proof}

We can now prove Lemma \ref{majexp}.

\begin{proof}[Proof of Lemma \ref{majexp}.] 
We need to study the asymptotic behaviour of the sequence of operators $(\beta_{n})_n$. Since $\Gamma_{m}$ admits a spectral decomposition, we have the upper bound $\| \alpha_{k}\| \leq \sup_{j}|1-\gamma_{k}\lambda_{j}|$ where $\left(\lambda_{j}\right)$ is the sequence of eigenvalues of $\Gamma_{m}$. Since for all $j \geq 1$ we have $0< c_{m}\leq \lambda_{j} \leq C$, there is a rank $n_{0}$ such that for all $n\geq n_{0}, \gamma_{n}C<1$. In particular, for all $n\geq n_{0}$ we have $\| \alpha_{n} \| \leq 1-\gamma_{n}c_{m}$. Thus, there is a positive constant $c_{1}$ such that for all $n\geq 1$:
\begin{equation}\label{majbeta1}
\| \beta_{n-1} \| \leq c_{1}\exp \left( -\lambda_{\min}\sum_{k=1}^{n-1}\gamma_{k}\right) \leq c_{1}\exp \left( -c_{m}\sum_{k=1}^{n-1}\gamma_{k}\right) ,
\end{equation}
where $\lambda_{\min}>0$ is the smallest eigenvalue of $\Gamma_{m}$ and $\| \beta_{n-1} \|$ denotes the spectral norm of operator $\beta_{n-1}$. Similarly, there is a positive constant $c_{2}$ such that for all integer $n$ and for all integer $k \leq n-1$:
\begin{equation}\label{majbeta2}
\left\| \beta_{n-1}\beta_{k}^{-1}\right\| \leq c_{2}\exp \left( -c_{m}\sum_{j=k+1}^{n-1}\gamma_{j} \right) .
\end{equation}
Moreover, for all $n > n_{0}$ and $k\geq n_{0}$ such that $k \leq n-1$,
\begin{equation}
\label{majbeta} \left\| \beta_{n-1}\beta_{k}^{-1}\right\| \leq \exp \left( -c_{m}\sum_{j=k+1}^{n-1}\gamma_{j}\right) ,
\end{equation}
see \cite{HC} for more details. Using decomposition (\ref{decbeta}) again, we get
\begin{equation}\label{maj3}
\mathbb{E}\left[ \| Z_{n}-m \|^{2}\right] \leq 3 \mathbb{E}\left[ \| \beta_{n-1}(Z_{1}-m)\|^{2}\right] +3\mathbb{E}\left[ \| \beta_{n-1}M_{n}\|^{2}\right] +3\mathbb{E}\left[ \| \beta_{n-1}R_{n} \|^{2}\right] .
\end{equation}
We now bound each term at the right-hand side of previous inequality. \\

\noindent \textit{Step 1: The quasi-deterministic term:} Using inequality (\ref{majbeta1}), with help of an integral test for convergence, for all $n \geq 1$:
\begin{align*}
\mathbb{E}\left[ \| \beta_{n-1} (Z_{1}-m ) \|^{2}\right] & \leq c_{1}^{2}\exp \left( -2c_{m}\sum_{k=1}^{n-1}\gamma_{k}\right) \mathbb{E}\left[ \| Z_{1}-m \|^{2} \right] \\
& \leq c_{1}^{2}\left( -2c_{m}c_{\gamma}\int_{1}^{n}t^{-\alpha}dt \right) \mathbb{E}\left[ \| Z_{1}-m \|^{2} \right] \\
& \leq c_{1}^{2}M\exp \left( 2\frac{c_{m}c_{\gamma}}{1-\alpha}\right)\exp\left( -2\frac{c_{m}c_{\gamma}}{1-\alpha}n^{1-\alpha}\right) .
\end{align*}
Since $\alpha <1$, this term converges exponentially fast to $0$.\\

\noindent \textit{Step 2: The martingale term:} We have
\begin{align*}
\| \beta_{n-1}M_{n}\|^{2} & = \left\| \sum_{k=1}^{n-1}\gamma_{k}\beta_{n-1}\beta_{k}^{-1}\xi_{k+1}\right\|^{2} \\
& = \sum_{k=1}^{n-1}\gamma_{k}^{2}\left\| \beta_{n-1}\beta_{k}^{-1}\xi_{k+1}\right\|^{2}+2\sum_{k=1}^{n-1}\sum_{k'<k}\gamma_{k}\gamma_{k'}\langle \beta_{n-1}\beta_{k}^{-1}\xi_{k+1},\beta_{n-1}\beta_{k'}^{-1}\xi_{k'+1}\rangle \\
& \leq \sum_{k=1}^{n-1}\gamma_{k}^{2}\left\| \beta_{n-1}\beta_{k}^{-1}\right\|^{2}\| \xi_{k+1}\|^{2}+2\sum_{k=1}^{n-1}\sum_{k'<k}\gamma_{k}\gamma_{k'}\langle \beta_{n-1}\beta_{k}^{-1}\xi_{k+1},\beta_{n-1}\beta_{k'}^{-1}\xi_{k'+1}\rangle .
\end{align*}
Since $\left( \xi_{n}\right)$ is a sequence of martingale differences, for all $k'<k$ we have 
\[\mathbb{E}\left[ \langle \xi_{k+1},\xi_{k'+1}\rangle \right]=\mathbb{E}\left[ \mathbb{E}\left[ \langle \xi_{k+1},\xi_{k'+1}\rangle |\mathcal{F}_{k}\right]\right]=\mathbb{E}\left[ \langle \mathbb{E}\left[ \xi_{k+1}|\mathcal{F}_{k}\right],\xi_{k'+1}\rangle \right] =0.\]
 Thus, 
\begin{equation}
\mathbb{E}\left[ \| \beta_{n-1}M_{n}\|^{2}\right] \leq \sum_{k=1}^{n-1}\gamma_{k}^{2}\left\| \beta_{n-1}\beta_{k}^{-1}\right\|^{2},
\end{equation}
because for all $k \in \mathbb{N}$, $\mathbb{E}\left[ \| \xi_{k+1}\|^{2}\right] \leq 1$. The term $\| \beta_{n-1}\beta_{k}^{-1}\|$ converges exponentially fast  to $0$ when $k$ is small enough compared to $n$. We denote by $E(.)$ the integer function and we isolate the term which gives the rate of convergence. Let us split the sum into two parts: 
\begin{equation}
\label{split}\sum_{k=1}^{n-1}\gamma_{k}^{2}\left\| \beta_{n-1}\beta_{k}^{-1}\right\|^{2} = \sum_{k=1}^{E(n/2)-1}\gamma_{k}^{2}\| \beta_{n-1}\beta_{k}^{-1}\|^{2} + \sum_{k=E(n/2)}^{n-1}\gamma_{k}^{2}\| \beta_{n-1}\beta_{k}^{-1}\|^{2}.
\end{equation}
We shall show that the first term on the right-hand side on (\ref{split}) converges exponentially fast to $0$ and that the second term on the right-hand side converges at the rate $\frac{1}{n^{\alpha}}$. Indeed, we deduce from inequality (\ref{majbeta2}):
\begin{align*}
\sum_{k=1}^{E(n/2)-1}\gamma_{k}^{2}\left\| \beta_{n-1}\beta_{k}^{-1}\right\|^{2} & \leq c_{2}\sum_{k=1}^{E(n/2)-1}\gamma_{k}^{2}e^{-2c_{m}\sum_{j=k+1}^{n-1}\gamma_{j}} \\
& \leq c_{2}\sum_{k=1}^{E(n/2)-1}\gamma_{k}^{2}e^{-2c_{m}\frac{n}{2}\frac{c_{\gamma}}{n^{\alpha}}}\\
& \leq c_{2}e^{-c_{m}c_{\gamma}n^{1-\alpha}}\sum_{k=1}^{E(n/2)-1}\gamma_{k}^{2}.
\end{align*}
Since $\sum \gamma_{k}^{2}<\infty$, we get 
\[\sum_{k=1}^{E(n/2)-1}\gamma_{k}^{2}\| \beta_{n-1}\beta_{k}^{-1}\|^{2}=O\left( e^{-c_{m}c_{\gamma}n^{1-\alpha}}\right).\]
We now bound the second term at the right-hand side of (\ref{split}). Using inequality (\ref{majbeta}), for all $n > 2n_{0}$:
\begin{align*}
\sum_{k=E(n/2)}^{n-1}\gamma_{k}^{2}\left\| \beta_{n-1}\beta_{k}^{-1}\right\|^{2} & \leq \sum_{k=E(n/2)}^{n-2}\gamma_{k}^{2}e^{-2c_{m}\sum_{j=k+1}^{n-1}\gamma_{j}} +\gamma_{n-1}^{2}\\
& \leq c_{\gamma}\left( \frac{1}{E(n/2)}\right)^{\alpha}\sum_{k=E(n/2)}^{n-2}\gamma_{k}e^{-2c_{m}\sum_{j=k+1}^{n-1}\gamma_{j}} +\gamma_{n-1}^{2} \\
& \leq \frac{2^{\alpha}c_{\gamma}}{n^{\alpha}}\sum_{k=E(n/2)}^{n-2}\gamma_{k}e^{-2c_{m}\sum_{j=k+1}^{n-1}\gamma_{j}}+ \gamma_{n-1}^{2}.
\end{align*}
Moreover, for all $n>2n_{0}$ and $k\leq n-2$:
\[\sum_{j=k+1}^{n-1}\gamma_{j}\leq \int_{k+1}^{n}\frac{c_{\gamma}}{s^{\alpha}}ds = \frac{c_{\gamma}}{1-\alpha}\left[ n^{1-\alpha}-(k+1)^{1-\alpha}\right] ,
\]
and hence $e^{-2c_{m}\sum_{j=k+1}^{n-1}\gamma_{j}}\leq e^{-2c_{m}\frac{c_{\gamma}}{1-\alpha}\left[ n^{1-\alpha}-(k+1)^{1-\alpha}\right]}$. Since $\frac{1}{k^{\alpha}}\leq \frac{2}{(k+1)^{\alpha}}$,
\begin{align*}
\sum_{k=E(n/2)}^{n-2}\gamma_{k}e^{2c_{m}\frac{c_{\gamma}}{1-\alpha}(k+1)^{1-\alpha}} & \leq 2^{\alpha}c_{\gamma}\sum_{k=E(n/2)}^{n-2}\frac{1}{(k+1)^{\alpha}}e^{2c_{m}\frac{c_{\gamma}}{1-\alpha}(k+1)^{1-\alpha}} \\
& \leq 2^{\alpha}c_{\gamma}\int_{E(n/2)}^{n-1}\frac{1}{(t+1)^{\alpha}}e^{2c_{m}\frac{c_{\gamma}}{1-\alpha}(t+1)^{1-\alpha}}dt \\
& \leq \frac{2^{\alpha}}{2c_{m}}e^{2c_{m}\frac{c_{\gamma}}{1-\alpha}n^{1-\alpha}} \\
& \leq \frac{2^{\alpha -1}}{c_{m}}e^{2c_{m}\frac{c_{\gamma}}{1-\alpha}n^{1-\alpha}}.
\end{align*}
Note that the integral test for convergence is valid because there is a rank $n_{0}\geq 1$ such that the function $t \longmapsto \frac{1}{(t+1)^{\alpha}}e^{2c_{m}\frac{c_{\gamma}}{1-\alpha}(t+1)^{1-\alpha}}$ is increasing on $[n_{0}',\infty )$. Let $n_{1}:=\max \lbrace 2n_{0}+1,n_{0}'\rbrace$, for all $n\geq n_{1}$:
\begin{equation}
\sum_{k=E(n/2)}^{n-1}\gamma_{k}^{2}\left\| \beta_{n-1}\beta_{k}^{-1}\right\|^{2} \leq \frac{2^{2\alpha -1}c_{\gamma}}{c_{m}}\frac{1}{n^{\alpha}} + c_{\gamma}2^{2\alpha}\frac{1}{n^{2\alpha}}.
\end{equation}
Consequently, there is a positive constant $C_{2}$ such that for all $n \geq 1$,
\begin{equation}\label{majmart}
3\mathbb{E}\left[ \| \beta_{n-1}M_{n}\|^{2}\right] \leq C_{2}\frac{1}{n^{\alpha}}.
\end{equation}
\newline
\textit{Step 3: The remainder term}. 
In the same way, we split the sum into two parts:
\begin{equation}
\label{splitdelta} \sum_{k=1}^{n-1}\gamma_{k}\beta_{n-1}\beta_{k}^{-1}\delta_{k} = \sum_{k=1}^{E(n/2)-1}\gamma_{k}\beta_{n-1}\beta_{k}^{-1}\delta_{k} + \sum_{k=E(n/2)}^{n-1}\gamma_{k}\beta_{n-1}\beta_{k}^{-1}\delta_{k}.
\end{equation}
It can be checked (see the proof of Lemma \ref{lemmajopourrie} for more details) that there is a positive constant $M$ such that for all $n\geq 1$,
\begin{equation}
\mathbb{E}\left[ \| Z_{n}-m \|^{4}\right] \leq M.
\end{equation}
Moreover, by Lemma \ref{lemdelta}, almost surely $\| \delta_{n}\| \leq C_{m}\| Z_{n}-m \|^{2}$. Thus, for all $k,k' \geq 1$, applying Cauchy-Schwarz's inequality, 
\begin{align*}
\mathbb{E} \left[ \| \delta_{k}\| \| \delta_{k'}\| \right] & \leq C_{m}^{2}\mathbb{E}\left[ \| Z_{k}-m \|^{2}\| Z_{k'}-m \|^{2} \right] \\
& \leq C_{m}^{2}\sqrt{ \mathbb{E}\left[ \| Z_{k}-m \|^{4} \right]}\sqrt{ \mathbb{E}\left[ \| Z_{k'}-m\|^{4}\right]} \\
& \leq C_{m}^{2}\sup_{n\geq 1}\mathbb{E}\left[ \| Z_{n}-m\|^{4}\right] \\
& \leq C_{m}^{2}M.
\end{align*}  
As a particular case, we also have $\mathbb{E}\left[ | \langle \delta_{k},\delta_{k'}\rangle | \right] \leq C_{m}^{2}M$. Applying this result to the term on the right-hand side of (\ref{splitdelta}),
\begin{align*}
\mathbb{E}\left[ \left\| \sum_{k=1}^{E(n/2)-1}\gamma_{k}\beta_{n-1}\beta_{k}^{-1}\delta_{k}\right\|^{2}\right] & \leq C_{m}^{2}M \left[ \sum_{k=1}^{E(n/2)-1}\gamma_{k}\| \beta_{n-1}\beta_{k}^{-1}\| \right]^{2} \\
& \leq c_{2}C_{m}^{2}Me^{-2c_{m}c_{\gamma}n^{1-\alpha}}\left( \sum_{k=1}^{E(n/2)-1}\gamma_{k}\right)^{2} \\
& \leq C_{1}'e^{-2c_{m}c_{\gamma}n^{1-\alpha}}n^{2-2\alpha}.
\end{align*}
This term converges exponentially fast to $0$. To bound the second term, we use the same idea as for the martingale term. Applying previous inequalities for the terms $\mathbb{E}\left[ \| \delta_{k}\| \| \delta_{k'}\| \right]$ which appear in the double products, we get:
\begin{align*}
\mathbb{E}\left[ \left\| \sum_{k=E(n/2)}^{n-1}\gamma_{k}\beta_{n-1}\beta_{k}^{-1}\delta_{k}\right\|^{2}\right] & \leq C_{m}^{2}\sup_{E(n/2)\leq k \leq n-1}\mathbb{E}\left[ \| Z_{k}-m \|^{4}\right] \left[ \sum_{k=E(n/2)}^{n-1}\gamma_{k}\| \beta_{n-1}\beta_{k}^{-1}\|\right]^{2} \\
& \leq C_{3}\sup_{E(n/2)\leq k \leq n-1}\mathbb{E}\left[ \| Z_{k}-m \|^{4}\right] ,
\end{align*}
since $\left[ \sum_{k=E(n/2)}^{n-1}\gamma_{k}\| \beta_{n-1}\beta_{k}^{-1}\|\right]^{2}$ is bounded. This fact can be checked with similar calculus to the ones in the proof of inequality (\ref{majmart}). 
\end{proof}
To prove Lemma \ref{propor4}, we introduce two technical Lemmas. The first one gives a sharp convexity bound 
when $\| Z_{n}-m \|$ is not too large.
\begin{lem}\label{lemmajznphi}
If assumptions \textbf{(A1)} and \textbf{(A2)} hold, there are a rank $n_{\alpha}$ and a constant $c$ such that for all $n \geq n_{\alpha}$, $\| Z_{n}-m \| \leq cn^{1-\alpha}$ yields
\begin{equation}\label{inegalitelemmepreuve}
\langle \Phi (Z_{n}),Z_{n}-m \rangle \geq \frac{1}{c_{\gamma}n^{1-\alpha}}\| Z_{n}-m \|^{2} .
\end{equation}
As a corollary, there is also a deterministic rank $n_{\alpha}'$ such that for all $n \geq n_{\alpha}'$, $\| Z_{n}-m \| \leq cn^{1-\alpha}$ yields
\begin{equation}
\label{ineglemzngammaphizn}\| Z_{n}-m - \gamma_{n}\Phi(Z_{n}) \|^{2} \leq \left( 1-\frac{1}{n}\right) \| Z_{n}-m \|^{2} .
\end{equation} 
\end{lem}

\begin{proof}[Proof of Lemma \ref{lemmajznphi}] We suppose that $\| Z_{n} - m \| \leq cn^{1-\alpha}$. We have to consider two cases. 

If $ \| Z_{n}-m \| \leq 1$,  we have  $\| Z_{n} \| \leq \| m \| +1$, so that by Corollary 2.2 in \cite{HC}, there is a positive constant $c_{1}$ such that $\langle \Phi(Z_{n}),Z_{n}-m \rangle \geq c_{1}\| Z_{n}-m \|^{2}$.

If now $\| Z_{n}-m \| \geq 1$, since $\Phi (Z_{n})= \int_{0}^{1} \Gamma_{m+t(Z_{n}-m)}(Z_{n}-m )dt$, by continuity and linearity of the inner product, 
\[
\langle \Phi (Z_{n}) ,Z_{n}-m \rangle = \int_{0}^{1}\left\langle Z_{n}-m, \Gamma_{m+t(Z_{n}-m)}(Z_{n}-m)\right\rangle dt. 
\]
Moreover, operators $\Gamma_{h}$ are non negative for all $h \in H$. Applying Proposition 2.1 of \cite{HC}, and since for all $t\in \left[ 0, \frac{1}{\| Z_{n}-m \|} \right]$ we have $\| m +t(Z_{n}-m) \| \leq \| m \| +1$, there is a positive constant $c_{2}$ such that:
\begin{align*}
\langle \Phi (Z_{n}),Z_{n}-m \rangle & = \int_{0}^{1} \left\langle Z_{n}-m , \Gamma_{m+t(Z_{n}-m )}(Z_{n}-m )\right\rangle dt \\
& \geq \int_{0}^{1/ \| Z_{n}-m \|} \left\langle Z_{n}-m , \Gamma_{m+t(Z_{n}-m )}(Z_{n}-m )\right\rangle dt \\
& \geq \int_{0}^{1/ \| Z_{n}-m \|}c_{2} \| Z_{n}-m \|^{2} dt \\
& = \frac{c_{2}}{\| Z_{n}-m \|}\| Z_{n}-m \|^{2} \\
& \geq \frac{c_{2}}{cn^{1-\alpha}}\| Z_{n}-m \|^{2} .
\end{align*}
We can choose a rank $n_{\alpha}$ such that for all $n\geq n_{\alpha}$ we have $c_{1} \geq \frac{1}{c_{\gamma}n^{1-\alpha}}$ which concludes the proof of inequality (\ref{inegalitelemmepreuve}) with $c=c_{2}c_{\gamma}$.\\
\newline
We now prove inequality (\ref{ineglemzngammaphizn}). For all $n \geq n_{\alpha}$, $\| Z_{n}-m \| \leq cn^{1-\alpha}$ yields
\begin{align*}
 \| Z_{n}-m - \gamma_{n} \Phi(Z_{n}) \|^{2} & = \| Z_{n}-m \|^{2} -2\gamma_{n} \langle \Phi (Z_{n}) , Z_{n}-m \rangle + \gamma_{n}^{2} \| \Phi (Z_{n}) \|^{2} \\
 & \leq \| Z_{n}-m \|^{2} - \frac{2}{c_{\gamma}n^{1-\alpha}}\frac{c_{\gamma}}{n^{\alpha}} \| Z_{n}-m \|^{2} + \gamma_{n}^{2}C^{2}\| Z_{n}-m \|^{2} \\
 & = \left( 1- \frac{2}{n} + C^{2}\frac{c_{\gamma}^{2}}{n^{2 \alpha}} \right) \| Z_{n}-m \|^{2}. 
\end{align*}
Consequently, we can choose a rank $n_{\alpha }' \geq n_{\alpha}$ such that for all $n \geq n_{\alpha}'$ we have $C^{2}c_{\gamma}^{2}n^{-2\alpha} \leq n^{-1}$. Note that this is possible since $\alpha > 1/2$. 
\end{proof}
\begin{lem}\label{lemmajopourrie}
There is a positive constant $C_{\alpha}$ such that for all $n \geq 1$,
\[
 \mathbb{P}\left[ \| Z_{n}-m \| \geq cn^{1-\alpha} \right] \leq \frac{C_{\alpha}}{n^{4-\alpha}}, 
 \]
where constant $c$ has been defined in Lemma~\ref{lemmajznphi}.
\end{lem}

\begin{proof}[Proof of Lemma \ref{lemmajopourrie}]
In order to use Markov's inequality, we prove by induction that for all integer $p \geq 1$, there is a positive constant $M_{p}$ such that for all $n$:
\[
 \mathbb{E}\left[ \| Z_{n} -m \|^{2p} \right] \leq M_{p}.
\]
\cite{HC} have proved previous inequality for the particular case $p=1$. Decomposition (\ref{decphi}) yields
\begin{align*}
\| Z_{n+1}-m \|^{2} & = \| Z_{n}-m \|^{2} +\gamma_{n}^{2}\| \Phi (Z_{n}) \|^{2} + \gamma_{n}^{2} \| \xi_{n+1}\|^{2} \\
&   -2\gamma_{n}\langle Z_{n} -m , \Phi (Z_{n} ) \rangle -2 \gamma_{n}^{2} \langle \xi_{n+1}, \Phi (Z_{n} ) \rangle +2\gamma_{n} \langle \xi_{n+1} , Z_{n}-m \rangle .
\end{align*}
Moreover, $\langle \xi_{n+1},\Phi (Z_{n}) \rangle = - \langle U_{n+1},\Phi(Z_{n}) \rangle + \| \Phi (Z_{n})\|^{2}$. Since $\| \Phi (Z_{n}) \| \leq 1$, $\| \xi_{n+1} \| \leq 2$ and $\langle \Phi (Z_{n}) , Z_{n}-m \rangle \geq 0$, applying Cauchy-Schwarz's inequality, for all $n \geq 1$
\begin{equation}
\label{decrecmoment} \| Z_{n+1}-m \|^{2} \leq \| Z_{n}-m \|^{2} + 6\gamma_{n}^{2} + 2 \gamma_{n}\langle \xi_{n+1}, Z_{n}-m \rangle .
\end{equation}
Using previous  inequality,
\begin{align}
\notag \| Z_{n+1}-m \|^{2p} & \leq \left( \| Z_{n}-m \|^{2}+6\gamma_{n}^{2}+2\gamma_{n} \langle \xi_{n+1},Z_{n}-m-\gamma_{n}\Phi (Z_{n}) \rangle \right)^{p} \\
\notag& = \sum_{k=0}^{p}\binom{p}{k}\left(2\gamma_{n}\langle \xi_{n+1},Z_{n}-m \rangle \right)^{k}\left( \| Z_{n}-m \|^{2}+6\gamma_{n}^{2}\right)^{p-k} \\
\label{decznp} & = \left( \| Z_{n}-m \|^{2}+6\gamma_{n}^{2}\right)^{p} + 2p\gamma_{n}\langle \xi_{n+1},Z_{n}-m \rangle \left( \| Z_{n}-m \|^{2}+6\gamma_{n}^{2}\right)^{p-1} \\
\notag & + \sum_{k=2}^{p}\binom{p}{k}\left(2\gamma_{n}\langle \xi_{n+1},Z_{n}-m \rangle \right)^{k}\left( \| Z_{n}-m \|^{2}+6\gamma_{n}^{2}\right)^{p-k}.
\end{align}
We now bound the three terms in (\ref{decznp}). First, using induction assumptions,
\begin{align*}
\mathbb{E}\left[\left( \| Z_{n}-m \|^{2}+6\gamma_{n}^{2}\right)^{p}\right] & = \mathbb{E}\left[ \| Z_{n}-m \|^{2p} + \sum_{k=0}^{p-1}\binom{p}{k} \| Z_{n}-m \|^{2k} (6\gamma_{n}^{2})^{p-k} \right] \\
& =  \mathbb{E}\left[ \| Z_{n}-m \|^{2p} \right] + \sum_{k=0}^{p-1}\binom{p}{k} \mathbb{E}\left[ \| Z_{n}-m \|^{2k} \right] (6\gamma_{n}^{2})^{p-k} \\
& \leq  \mathbb{E}\left[ \| Z_{n}-m \|^{2p} \right] + \sum_{k=0}^{p-1}\binom{p}{k} M_{k} (6\gamma_{n}^{2})^{p-k}.
\end{align*}
Since for all $k \leq p-1$ we have $\left(\gamma_{n}^{2}\right)^{p-k} = o\left(\gamma_{n}^{2}\right)$, there is a positive constant $C_{p}$ such that for all $n \geq 1$,
\begin{equation}
\mathbb{E}\left[\left( \| Z_{n}-m \|^{2}+6\gamma_{n}^{2}\right)^{p}\right]  \leq \mathbb{E}\left[ \| Z_{n} - m \|^{2p} \right] + C_{p}\gamma_{n}^{2}.
\end{equation}
Remark that $C_{p}$ does not depend on $n$. Let us now deal with the second term in (\ref{decznp}). Since $(\xi_{n})$ is a sequence of martingale differences adapted to the filtration $\left( \mathcal{F}_{n}\right)$ and since $Z_{n}$ is $\mathcal{F}_{n}$-measurable, for all $n \geq 1$, 
\[
 \mathbb{E}\left[ 2\gamma_{n}\langle \xi_{n+1},Z_{n}-m \rangle \left( \| Z_{n}-m \|^{2}+6\gamma_{n}^{2}\right)^{p-1} |\mathcal{F}_{n} \right] = 0.
\]
It remains to bound the last term in (\ref{decznp}). Applying Cauchy Schwarz's inequality,  we get, for all $n \geq 1$,
\begin{align*}
 \sum_{k=2}^{p}\binom{p}{k} & \mathbb{E}\left[ \left(2\gamma_{n}\langle \xi_{n+1},Z_{n}-m \rangle \right)^{k}\left( \| Z_{n}-m \|^{2}+6\gamma_{n}^{2}\right)^{p-k} \right] \\
 & = \sum_{k=2}^{p}\binom{p}{k} \mathbb{E}\left[ \left(2\gamma_{n}\langle \xi_{n+1},Z_{n}-m \rangle \right)^{k}\sum_{j=0}^{p-k}\binom{p-k}{j}\left(\|Z_{n}-m \|^{2}\right)^{p-k-j}\left(6\gamma_{n}^{2}\right)^{j} \right] \\
 & = \sum_{k=2}^{p}\sum_{j=0}^{p-k}\binom{p-k}{j}\binom{p}{k} 2^{k+j}3^{j}\gamma_{n}^{k+2j}\mathbb{E}\left[ \left(\langle \xi_{n+1},Z_{n}-m \rangle \right)^{k}\|Z_{n}-m \|^{2(p-k-j)} \right] \\
 & \leq \sum_{k=2}^{p}\sum_{j=0}^{p-k}\binom{p-k}{j}\binom{p}{k} 2^{k+j}3^{j}\gamma_{n}^{k+2j}\mathbb{E}\left[ \left\| \xi_{n+1} \right\|^{k}\left\|Z_{n}-m \right\|^{2p-k-2j}  \right] .
 \end{align*}
Since $\| \xi_{n+1} \| \leq 2$, we have for all $n\geq 1$,   
 \begin{align*}
\sum_{k=2}^{p}& \sum_{j=0}^{p-k}\binom{p-k}{j}\binom{p}{k} 2^{k+j}3^{j}\gamma_{n}^{k+2j}\mathbb{E}\left[ \left\| \xi_{n+1} \right\|^{k}\left\|Z_{n}-m \right\|^{2p-k-2j}  \right] \\
 & \leq \sum_{k=2}^{p}\sum_{j=0}^{p-k}\binom{p-k}{j}\binom{p}{k} 2^{2k+j}3^{j}\gamma_{n}^{k+2j}\mathbb{E}\left[ \left\|Z_{n}-m \right\|^{2p-k-2j}  \right] .
 \end{align*}
Finally, using Cauchy-Schwarz's inequality and by induction, we get
 \begin{align*}
\sum_{k=2}^{p} & \sum_{k=2}^{p}\sum_{j=0}^{p-k}\binom{p-k}{j}\binom{p}{k} 2^{2k+j}3^{j}\gamma_{n}^{k+2j}\mathbb{E}\left[ \left\|Z_{n}-m \right\|^{2p-k-2j}  \right] \\
& \leq \sum_{k=2}^{p}\sum_{j=0}^{p-k}\binom{p-k}{j}\binom{p}{k}2^{2k+j}3^{j}\gamma_{n}^{k+2j}\sqrt{\mathbb{E}\left[  \|Z_{n}-m \|^{2(p-1-j)}\right] }\sqrt{ \mathbb{E}\left[ \| Z_{n}-m \|^{2(p-k-j+1)} \right]} \\
 & \leq \sum_{k=2}^{p}\sum_{j=0}^{p-k}\binom{p-k}{j}\binom{p}{k}2^{2k+j}3^{j}\gamma_{n}^{k+2j} \sqrt{M_{p-1-j}}\sqrt{ M_{p-k-j+1} }.
\end{align*}
Moreover, for all $k \geq 2$ and $j \geq 0$, $\gamma_{n}^{2j+k}=O (\gamma_{n}^{2})$, so there is a constant $C_{p}'$ such that for all $n \geq 1$:
\begin{equation}
\notag \sum_{k=2}^{p}\binom{p}{k}  \mathbb{E}\left[ \left(2\gamma_{n}\langle \xi_{n+1},Z_{n}-m \rangle \right)^{k}\left( \| Z_{n}-m \|^{2}+6\gamma_{n}^{2}\right)^{p-k} \right] \leq C_{p}' \gamma_{n}^{2}
\end{equation}
Remark that $C_{p}'$ does not depend on $n$. Since $Z_ {1}$ is bounded by construction, we get by induction 
\begin{align*}
\mathbb{E}\left[ \| Z_{n+1}-m \|^{2p}\right] & \leq \mathbb{E}\left[ \| Z_{n}-m \|^{2p}\right] + \left(C_{p}+C_{p}' \right)\gamma_{n}^{2} \\
& \leq \mathbb{E}\left[ \| Z_{1}-m \|^{2p} \right] + \left( C_{p}+C_{p}'\right) \sum_{k=1}^{n}\gamma_{k}^{2} \\
&  \leq \mathbb{E}\left[ \| Z_{1}-m \|^{2p} \right] + \left( C_{p}+C_{p}'\right) \sum_{k=1}^{\infty}\gamma_{k}^{2} \\
& \leq M_{p},
\end{align*}
which concludes the induction.

Applying Markov's inequality, for all integer $p \geq 1$:
\[
\mathbb{P}\left[ \| Z_{n}-m \| \geq cn^{1-\alpha} \right] \leq \frac{\mathbb{E}\left[ \| Z_{n}-m \|^{2p}\right]}{(cn^{1-\alpha})^{2p}} \leq \frac{M_{p}}{(cn^{1-\alpha})^{2p}}.
\]
The announced result is obtained  by considering $p \geq \frac{4-\alpha}{2(1-\alpha)}$. Note that this is possible since $\alpha \neq 1$.
\end{proof}

\begin{proof}[Proof of Lemma \ref{propor4}.] For all $n \geq 1$,
\begin{equation}
\label{eqznm4=} \mathbb{E}\left[ \| Z_{n+1}-m \|^{4} \right] = \mathbb{E}\left[ \| Z_{n+1}-m \|^{4} \mathbb{1}_{\| Z_{n}-m \| \geq cn^{1-\alpha}}\right] +  \mathbb{E}\left[ \| Z_{n+1}-m \|^{4} \mathbb{1}_{\| Z_{n}-m \| < cn^{1-\alpha}}\right] ,
\end{equation}
where constant $c$ has been defined in Lemma \ref{lemmajznphi}. Let us bound the first term at the right-hand side of~(\ref{eqznm4=}). Since $\| Z_{n+1}-m \| \leq \| Z_{n}-m \| + \gamma_{n} \leq \| Z_{1}-m \| + \sum_{k=1}^{n}\gamma_{k}$ and since $Z_{1}$ is bounded, there is a constant $C_{\alpha}'$ such that for all integer $n \geq 1$,
\[
 \| Z_{n}-m \| \leq C_{\alpha}'n^{1-\alpha}.
\] 
Consequently, 
\begin{align*}
\mathbb{E} \left[ \| Z_{n+1}-m \|^{4}\mathbb{1}_{\| Z_{n}-m \| \geq cn^{1-\alpha}} \right] & \leq \mathbb{E}\left[ \left( C_{\alpha}'(n+1)^{1-\alpha}\right)^{4} \mathbb{1}_{\| Z_{n}-m \| \geq cn^{1-\alpha}}\right] \\
& \leq \left( C_{\alpha}'(n+1)^{1-\alpha}\right)^{4}\mathbb{P}\left[ \| Z_{n}-m \| \geq cn^{1-\alpha}\right] . 
\end{align*}
Thus, applying Lemma \ref{lemmajopourrie}, we get
\begin{align*}
\left( C_{\alpha}'(n+1)^{1-\alpha}\right)^{4}\mathbb{E}\left[ \| Z_{n}-m \| \geq cn^{1-\alpha}\right] &  \leq \frac{C_{\alpha}'^{4}C_{\alpha}(n+1)^{4-4\alpha}}{n^{4-\alpha}}\\
& \leq 2^{4-4\alpha}\frac{C_{\alpha}'^{4}C_{\alpha}n^{4-4\alpha}}{n^{4-\alpha}} \\
& \leq 2^{4-4\alpha}\frac{C_{\alpha}'^{4}C_{\alpha}}{n^{3\alpha}}.
\end{align*}
We now bound the second term. Suppose that $\| Z_{n}-m \| \leq cn^{1-\alpha}$. Since $\| \xi_{n+1}\| \leq 2$, using Lemma \ref{lemmajznphi}, there is a rank $n_{\alpha}$ such that for all $n \geq n_{\alpha}$,
\begin{align*}
\| Z_{n+1}-m \|^{2}& \mathbb{1}_{ \| Z_{n}-m \| < cn^{1-\alpha}}\\
 & = \left( \| Z_{n}-m-\gamma_{n}\Phi (Z_{n}) \|^{2} + \gamma_{n}^{2} \| \xi_{n+1} \|^{2} +2\gamma_{n} \langle Z_{n}-m-\gamma_{n}\Phi (Z_{n}), \xi_{n+1}\rangle \right)\mathbb{1}_{\left\lbrace \| Z_{n}-m \| < cn^{1-\alpha}\right\rbrace}\\
& \leq \left( \left( 1-\frac{1}{n} \right) \| Z_{n}-m \|^{2} + 4\gamma_{n}^{2} + 2\gamma_{n} \langle Z_{n}-m-\gamma_{n}\Phi (Z_{n}), \xi_{n+1}\rangle \right) \mathbb{1}_{\left\lbrace \| Z_{n}-m \| < cn^{1-\alpha}\right\rbrace}.
\end{align*}
Moreover, since $(\xi_{n+1})$ is a sequence of martingale differences for the filtration $\left( \mathcal{F}_{n}\right)$, 
\begin{align*}
\mathbb{E}\left[ \langle Z_{n}-m-\gamma_{n}\Phi (Z_{n}), \xi_{n+1}\rangle \mathbb{1}_{\| Z_{n}-m \| \leq cn^{1-\alpha}} |\mathcal{F}_{n} \right] & =0,\\
\mathbb{E}\left[ \langle Z_{n}-m-\gamma_{n}\Phi (Z_{n}), \xi_{n+1}\rangle \| Z_{n}-m \|^{2}\mathbb{1}_{\| Z_{n}-m \| \leq cn^{1-\alpha}} |\mathcal{F}_{n} \right] & =0. 
\end{align*}
Applying Cauchy-Schwarz's inequality, 
\begin{align*}
\mathbb{E}\left[ \| Z_{n+1}-m \|^{4}\mathbb{1}_{\| Z_{n}-m \| \leq cn^{1-\alpha}}\right] & \leq \left( 1-\frac{1}{n}\right)^{2} \mathbb{E}\left[ \| Z_{n}-m \|^{4}\mathbb{1}_{\| Z_{n}-m \| \leq cn^{1-\alpha}}\right] +16 \gamma_{n}^{4}\\
&  + 8\gamma_{n}^{2}\left( 1-\frac{1}{n}\right) \mathbb{E}\left[ \| Z_{n}-m \|^{2}\mathbb{1}_{\| Z_{n}-m \| \leq cn^{1-\alpha}}\right] \\
& + 4\gamma_{n}^{2}\mathbb{E}\left[ \langle Z_{n}-m-\gamma_{n}\Phi (Z_{n}), \xi_{n+1}\rangle^{2} \mathbb{1}_{\| Z_{n}-m \| \leq cn^{1-\alpha}} \right] \\
& \leq \left( 1-\frac{1}{n}\right)^{2} \mathbb{E}\left[ \| Z_{n}-m \|^{4} \right] + 16\gamma_{n}^{4} + 8\gamma_{n}^{2}\mathbb{E}\left[ \| Z_{n}-m  \|^{2}\right] \\
& + 4\gamma_{n}^{2}\mathbb{E}\left[ \| Z_{n}-m -\gamma_{n}\Phi(Z_{n})\|^{2} \mathbb{E}\left[ \| \xi_{n+1} \|^{2}|\mathcal{F}_{n} \right]\mathbb{1}_{\| Z_{n}-m \| \leq cn^{1-\alpha}}  \right] .
\end{align*}
Finally, since $\mathbb{E} \left[ \| \xi_{n+1}\|^{2} | \mathcal{F}_{n} \right] \leq 1$, we get with Lemma \ref{lemmajopourrie},
\begin{align*}
\mathbb{E}\left[ \| Z_{n+1}-m \|^{4}\mathbb{1}_{\| Z_{n}-m \| \leq cn^{1-\alpha}}\right]& \leq \left( 1-\frac{1}{n}\right)^{2} \mathbb{E}\left[ \| Z_{n}-m \|^{4} \right] + 16\gamma_{n}^{4} + 8\gamma_{n}^{2}\mathbb{E}\left[ \| Z_{n}-m  \|^{2}\right] \\
& + 4 \gamma_{n}^{2}\left( 1-\frac{1}{n}\right) \mathbb{E}\left[ \| Z_{n}-m \|^{2}\right] \\
& \leq \left( 1-\frac{1}{n}\right)^{2} \mathbb{E}\left[ \| Z_{n}-m \|^{4} \right] + 16\gamma_{n}^{4} + 12\gamma_{n}^{2}\mathbb{E}\left[ \| Z_{n}-m  \|^{2}\right] .
\end{align*}
Since $\gamma_{n}^{4}=o \left( \frac{1}{n^{3\alpha}} \right)$, there are two positive constants $C_{1}'$ and $C_{2}'$ such that for all $n \geq n_{\alpha}$,
\begin{align*}
\mathbb{E}\left[ \| Z_{n+1}-m \|^{4} \right] & = \mathbb{E}\left[ \| Z_{n+1}-m \|^{4} \mathbb{1}_{\| Z_{n}-m \| \geq cn^{1-\alpha}}\right] +  \mathbb{E}\left[ \| Z_{n+1}-m \|^{4} \mathbb{1}_{\| Z_{n}-m \| \leq cn^{1-\alpha}}\right] \\
& \leq \frac{2^{4-4\alpha}C_{\alpha}'^{4}C_{\alpha}}{n^{3\alpha}} + \left( 1-\frac{1}{n}\right)^{2} \mathbb{E}\left[ \| Z_{n}-m \|^{4}\right] + 16\gamma_{n}^{4} + 12\gamma_{n}^{2}\mathbb{E}\left[ \| Z_{n}-m  \|^{2}\right] \\
& \leq \left( 1-\frac{1}{n} \right)^{2}\mathbb{E}\left[ \| Z_{n}-m \|^{4}\right] + C_{1}'\frac{1}{n^{3\alpha}}+C_{2}'\frac{1}{n^{2\alpha}}\mathbb{E}\left[ \| Z_{n}-m \|^{2}\right] .
\end{align*}
\end{proof}
\begin{proof}[Proof of Theorem \ref{bonnevitesse}]
Let $\beta \in ( \alpha , 3\alpha -1)$. Let us check that there is a rank $n_{\beta}\geq n_{\alpha}$ ($n_{\alpha}$ has been defined in Lemma \ref{propor4})  such that for all $n\geq n_{\beta}$, we have 
\[\left( 1- \frac{1}{n}\right)^{2}\left( \frac{n+1}{n}\right)^{\beta} + \left( C_{1}' + C_{2}' \right)2^{3\alpha}\frac{1}{(n+1)^{3\alpha - \beta}} \leq 1,\] 
($C_{1}',C_{2}'$ are defined in Lemma \ref{propor4}). Indeed, since $\beta < 3\alpha -1 < 2$, 
\begin{align*}
\left( 1- \frac{1}{n}\right)^{2}\left( \frac{n+1}{n}\right)^{\beta} + \left( C_{1}' + C_{2}' \right)2^{3\alpha}\frac{1}{(n+1)^{3\alpha - \beta}} & = \left( 1-\frac{2}{n}+o\left( \frac{1}{n}\right) \right)\left( 1+\frac{\beta}{n}+o\left( \frac{1}{n}\right) \right) + o\left( \frac{1}{n}\right) \\
& = 1- (2-\beta)\frac{1}{n}+o\left( \frac{1}{n}\right) .
\end{align*}
We now prove by induction that there are two positive constants $C'$ and $C''$ such that $2C' \geq C''\geq C' \geq 1$ and such that for all $n \geq n_{\beta}$,
\begin{align*}
\mathbb{E}\left[ \| Z_{n}-m \|^{2}\right]  & \leq \frac{C'}{n^{\alpha}} \\
\mathbb{E}\left[ \| Z_{n}-m \|^{4}\right] & \leq \frac{C''}{n^{\beta}} .
\end{align*}
Let us choose $C'\geq n_{\beta}\mathbb{E}\left[ \| Z_{n_{\beta}}-m \|^{2}\right] $ and $C''\geq n_{\beta}\mathbb{E}\left[ \| Z_{n_{\beta}}-m \|^{4}\right]$. This is possible since there is a positive constant $M$ such that for all $n\geq 1$, $\mathbb{E}\left[ \| Z_{n}-m \|^{2} \right] \leq M$ and $\mathbb{E}\left[ \left\| Z_{n}-m \right\|^{4}\right] \leq M$. Let $n \geq n_{\beta}$, using Lemma \ref{propor4} and by induction,
\begin{align*}
\mathbb{E}\left[ \| Z_{n+1}-m \|^{4} \right] & \leq \left( 1-\frac{1}{n}\right)^{2}\mathbb{E}\left[ \| Z_{n}-m \|^{4} \| \right] + \frac{C_{1}'}{n^{3\alpha}}+\frac{C_{2}'}{n^{2\alpha}}\mathbb{E}\left[ \| Z_{n}-m \|^{2}\right] \\
& \leq \left( 1-\frac{1}{n}\right)^{2}\frac{C''}{n^{\beta}}+\frac{C_{1}'}{n^{3\alpha}}+\frac{C_{2}'C'}{n^{3\alpha}}.
\end{align*}
Moreover, since $C' \leq C''$ and since $C'' \geq 1$,
\begin{align*}
\mathbb{E}\left[ \| Z_{n+1}-m \|^{4} \right] & \leq \left( 1-\frac{1}{n}\right)^{2}\frac{C''}{n^{\beta}}+\frac{C_{1}'C''}{n^{3\alpha}}+\frac{C_{2}'C''}{n^{3\alpha}}. 
\end{align*}
Factorizing by $\frac{C''}{(n+1)^{\beta}}$, we get
\begin{align*}
\mathbb{E}\left[ \| Z_{n+1}-m \|^{4} \right] & \leq \left( 1- \frac{1}{n}\right)^{2}\left( \frac{n+1}{n}\right)^{\beta}\frac{C''}{(n+1)^{\beta}} + \left( C_{1}' + C_{2}' \right)\left( \frac{n+1}{n}\right)^{3\alpha}\frac{1}{(n+1)^{3\alpha - \beta}}\frac{C''}{(n+1)^{\beta}} \\
& \leq \left( 1- \frac{1}{n}\right)^{2}\left( \frac{n+1}{n}\right)^{\beta}\frac{C''}{(n+1)^{\beta}} + \left( C_{1}' + C_{2}' \right)2^{3\alpha}\frac{1}{(n+1)^{3\alpha - \beta}}\frac{C''}{(n+1)^{\beta}} \\
& \leq \left( \left( 1- \frac{1}{n}\right)^{2}\left( \frac{n+1}{n}\right)^{\beta} + \left( C_{1}' + C_{2}' \right)2^{3\alpha}\frac{1}{(n+1)^{3\alpha - \beta}} \right) \frac{C''}{(n+1)^{\beta}}.
\end{align*}
By definition of $n_{\beta}$, 
\begin{equation}
\mathbb{E}\left[ \| Z_{n+1}-m \|^{4} \right] \leq \frac{C''}{(n+1)^{\beta}}.
\end{equation}
We now prove that $\mathbb{E}\left[ \| Z_{n+1}-m \|^{2}\right] \leq \frac{C'}{(n+1)^{\alpha}}$. Since $C'' \leq 2C'$, by Lemma \ref{majexp} and by induction, there is a constant $C'''>0$ such that
\begin{align*}
\mathbb{E}\left[ \| Z_{n+1}-m \|^{2}\right] & \leq \frac{C'''}{(n+1)^{\alpha}}+C_{3}\sup_{n/2+1 \leq k \leq n+1}\mathbb{E}\left[ \| Z_{k}-m \|^{4}\right] \\
& \leq \frac{C'''}{(n+1)^{\alpha}}+2^{\beta}C_{3}\frac{C''}{(n+1)^{\beta}} \\
& \leq \frac{C'''}{(n+1)^{\alpha}}+ 2^{\beta +1}C_{3}\frac{1}{(n+1)^{\beta -\alpha}}\frac{C'}{(n+1)^{\alpha}}. 
\end{align*}
To get $\mathbb{E}\left[ \| Z_{n+1}-m \|^{2} \right] \leq \frac{C'}{(n+1)^{\alpha}} $, we only need to take $C' \geq C'''+ 2^{\beta +1}C_{3}\frac{1}{(n+1)^{\beta - \alpha}}$, which concludes the induction. 

The proof is complete for all $n \geq 1$ by taking $C'\geq \max_{n\leq n_{\beta}} \left\lbrace n^{\alpha}\mathbb{E}\left[ \| Z_{n}-m \|^{2}\right] \right\rbrace $ and $C''\geq \max_{n\leq n_{\beta}} \left\lbrace n^{\beta}\mathbb{E}\left[ \| Z_{n}-m \|^{4}\right] \right\rbrace $. 
\end{proof}
\begin{proof}[Proof of Proposition \ref{propbonnevitesse}] A lower bound for $\| Z_{n}-m-\Phi (Z_{n}) \|$ is obtained by using decomposition (\ref{decphi}). Using Corollary \ref{corgamma}, for all $h \in H$, 
\begin{align*}
\| \Phi (m+h) \| &  \leq \left\| \int_{0}^{1}\Gamma_{m+th}(h) dt \right\| \\
& \leq \int_{0}^{1}\left\| \Gamma_{m+th}(h) \right\| dt \\
& \leq C \| h \| .
\end{align*}
Consequently, there is a rank $n_{0}$ such that for all $n \geq n_{0}$,
\begin{align*}
\| h - \gamma_{n} \Phi (m+h) \| & \geq \left| \| h \| - \gamma_{n} \| \Phi (m+h) \| \right| \\
& \geq \| h \| - C \gamma_{n} \| h \| .
\end{align*}
In particular, for all $n \geq n_{0}$,
\[
\| Z_{n}-m-\gamma_{n}\Phi (Z_{n}) \| \geq \left( 1-C\gamma_{n} \right) \| Z_{n}-m \| . 
\]
Since $\lim_{n\rightarrow \infty}\mathbb{E}\left[ \| Z_{n}-m \|^{2}\right] = 0$, there is a rank $n_{0}'$ such that for all $n \geq n_{0}'$,
\[
\mathbb{E}\left[ \| \xi_{n+1}\|^{2} \right] = 1 -\mathbb{E}\left[ \| \Phi (Z_{n}) \|^{2}\right] \geq 1-C^{2}\mathbb{E}\left[ \| Z_{n}-m \|^{2} \right] . 
\]
Finally, since $\left( \xi_{n+1}\right)$ is a sequence of martingale differences adapted to the filtration $\left( \mathcal{F}_{n} \right)$, there is a rank $n_{1} \geq n_{0}'$ such that for all $n \geq n_{1}$,
\begin{align*}
\mathbb{E}\left[ \| Z_{n+1}-m \|^{2}\right] & \geq \left( 1-C\gamma_{n}\right)^{2} \mathbb{E}\left[ \| Z_{n}-m \|^{2}\right] + \gamma_{n}^{2}\left( 1-2C^{2}\mathbb{E}\left[ \| Z_{n}-m \|^{2}\right] \right) \\
& \geq \left( 1-2C\gamma_{n} \right) \mathbb{E}\left[ \| Z_{n}-m \|^{2} \right] + \gamma_{n}^{2}.
\end{align*}
We can prove by induction that there is a positive constant $C_{0}$ such that for all $n \geq n_{1}$,
\begin{align*}
 \mathbb{E}\left[ \| Z_{n}-m \|^{2} \right] &\geq \frac{C_{0}}{n^{\alpha}}.
\end{align*}
To conclude the proof, we just have to consider $C':= \min \left\lbrace \min_{1\leq n \leq n_{1}} \left\lbrace \mathbb{E}\left[ \| Z_{n}-m \|^{2}\right] n^{\alpha} \right\rbrace , C_{0} \right\rbrace $.
\end{proof}

\subsection{Proofs of the results given in Section \ref{sectionpinelis}}
\begin{proof}[Proof of Proposition \ref{pinelisadapte}] As in \cite{Pinelis}, let us define, for all integers $j$ and $n$ such that $2\leq j \leq n$, 
\begin{align*}
f_{j,n} & := \sum_{k=1}^{j-1}\gamma_{k}\beta_{n-1,k}\xi_{k+1} , \\
d_{j,n} & := f_{j,n} - f_{j-1,n} = \beta_{n-1,j-1}\gamma_{j-1}\xi_{j}, \\
e_{j,n} & := \mathbb{E}\left[ e^{\| d_{j,n} \| }-1-\| d_{j,n}\| |\mathcal{F}_{j-1}\right] ,
\end{align*} 
with $f_{0,n}=0$. Remark that for all $k \leq n-1$,
\[
\mathbb{E}\left[ \beta_{n-1,k}\xi_{k+1}|\mathcal{F}_{k}\right] = 0.
\]
It is not possible to apply directly Theorem 3.1 of \cite{Pinelis} because the sequence $\left( \beta_{n-1,k}\xi_{k+1} \right)$ is not properly a martingale differences sequence. As in \cite{Pinelis}, for all $t \in [0,1]$, let us define $u(t):=\| x+tv \|$, with $x,v \in H$. We have for all $t \in [0,1]$, $u'(t) \leq \| v \| $ and $\left( u^{2}(t)\right) '' \leq 2\| v \|^{2}$. Moreover, since for all $u \in \mathbb{R}$, $\cosh u \geq \sinh u$, we also get 
\begin{align*}
\left( \cosh u \right)''(t) &\leq \| v \|^{2} \cosh u. 
\end{align*}
Let $\varphi (t) := \mathbb{E} \left[ \cosh \left( \| f_{j-1,n}+td_{j,n} \| \right) |\mathbb{F}_{j}\right] $, 
\begin{align*}
\varphi ''(t) & \leq \mathbb{E}\left[ \| d_{j,n}\|^{2}\cosh \left( \| f_{j-1,n}+td_{j,n}\| \right) | \mathcal{F}_{j-1}\right] \\
& \leq \mathbb{E}\left[ \| d_{j,n}\|^{2}e^{t\| d_{j,n)} \|}\cosh \left( \| f_{j-1,n}\| \right) |\mathcal{F}_{j-1}\right] .
\end{align*}
Moreover, since $(\xi_{n})$ is a sequence of martingale differences adapted to the filtration $\left( \mathcal{F}_{n} \right)$, for all $j \geq 1$, $\mathbb{E}\left[ d_{j,n}|\mathcal{F}_{j-1}\right]=0$ and $\varphi '(0) = 0$. We get for all $j \geq 1$ such that $j \leq n$,
\begin{align*}
\mathbb{E}\left[ \cosh \left( \| f_{j,n}\| \right) |\mathcal{F}_{j-1} \right] & = \varphi (1) \\
& = \varphi (0) + \int_{0}^{1}(1-t)\varphi ''(t) dt \\
& \leq (1+e_{j,n})\cosh \left( \| f_{j-1,n} \| \right) .
\end{align*}
Let $G_{1}:=1$ and for all $2 \leq j \leq n$, let $G_{j}:=\frac{\cosh (\| f_{j,n} \| )}{\prod_{i=2}^{j}(1+e_{i,n})}$. Using previous inequality, since $e_{j+1,n}$ is $\mathcal{F}_{j}$-measurable, 
\begin{align*}
\mathbb{E}\left[ G_{j+1}|\mathcal{F}_{j} \right] & = \mathbb{E}\left[ \frac{\cosh \left( \left\| f_{j+1,n}\right\| \right) }{\prod_{i=2}^{j+1}\left( 1+e_{i,n}\right)}\big| \mathcal{F}_{j}\right] \\
& = \frac{\mathbb{E}\left[ \cosh \left( \left\| f_{j+1,n}\right\| \right) \big| \mathcal{F}_{j} \right] }{\prod_{i=2}^{j+1}\left( 1+e_{i,n}\right)} \\
& \leq \frac{\left(1+e_{j+1,n}\right)\cosh \left( \left\| f_{j,n}\right\| \right) }{\prod_{i=2}^{j+1}\left( 1+e_{i,n}\right)} \\
& = G_{j}.
\end{align*}
By induction, $\mathbb{E}\left[ G_{n} \right] \leq \mathbb{E} \left[ G_{1} \right] \leq 1$. Finally,
\begin{align*}
\mathbb{P}\left[ \| f_{n,n} \| \geq r \right] & \leq \mathbb{P}\left[ G_{n} \geq \frac{\cosh r}{\left\| \prod_{j=2}^{n}(1+e_{j,n}) \right\|}\right] \\
& \leq \mathbb{P}\left[ G_{n} \geq \frac{1}{2}\frac{\exp(r)}{\left\| \prod_{j=2}^{n}(1+e_{j,n}) \right\|}\right] \\
& \leq 2 \mathbb{E}\left[ G_{n} \right] e^{-r}\left\| \prod_{j=2}^{n}(1+e_{j,n}) \right\| \\
& \leq 2e^{-r}\left\| \prod_{j=2}^{n}(1+e_{j,n}) \right\| .
\end{align*}
\end{proof}
\begin{proof}[Proof of Theorem \ref{interv}] Using Theorem \ref{bonnevitesse}, one can check that $\mathbb{E}\left[ \| \beta_{n-1}R_{n} \| \right] = O\left( \frac{1}{n^{\alpha}}\right)$. Indeed, applying Lemma \ref{lemdelta}, 
\begin{align*}
\mathbb{E}\left[ \| \beta_{n-1}R_{n}\| \right] & \leq \sum_{k=1}^{n-1}\gamma_{k}\| \beta_{n-1}\beta_{k}^{-1}\| \mathbb{E}\left[ \| \delta_{k} \| \right] \\
& \leq C_{m} \sum_{k=1}^{n-1}\gamma_{k}\| \beta_{n-1}\beta_{k}^{-1}\| \mathbb{E}\left[ \| Z_{k}-m \|^{2}\right] .
\end{align*}
Moreover, with calculus similar to the ones for the upper bound of the martingale term in the proof of Proposition \ref{majexp} and applying Proposition \ref{propbonnevitesse}, 
\begin{align*}
C_{m} \sum_{k=1}^{n-1}\gamma_{k}\| \beta_{n-1}\beta_{k}^{-1}\| \mathbb{E}\left[ \| Z_{k}-m \|^{2}\right] & \leq C_{m}C'c_{\gamma}\sum_{k=1}^{n-1}\frac{1}{k^{2\alpha}}\| \beta_{n-1}\beta_{k}^{-1}\| \\
& = O \left( \frac{1}{n^{\alpha}}\right) . 
\end{align*} 
Finally, the term $ \beta_{n-1}(Z_{1}-m)$ converges exponentially to $0$. So, there are positive constants $C_{1},C_{1}',C_{2}$ such that 
\begin{equation}
\label{majinterv} \mathbb{P}\left[ \| Z_{n}-m \| \geq t \right] \leq \mathbb{P}\left[ \| \beta_{n-1}M_{n} \| \geq \frac{t}{2}\right] + \frac{C_{1}e^{-C_{1}'n^{1-\alpha}}}{t^{2}}+\frac{C_{2}}{n^{\alpha}}\frac{1}{t}
\end{equation}
We now give a "good" choice of sequences $( N_{n})$ and $(\sigma_{n}^{2})$ to apply Corollary \ref{corpinbern}.

\noindent \textit{Step 1: Choice of $N_{n}$.} \\
Using inequality (\ref{majbeta2}) and since $\| \xi_{n+1} \| \leq 2$, we have $\| \beta_{n-1}\beta_{k}^{-1}\xi_{k+1} \| \leq 2c_{2}\gamma_{k}e^{-\lambda_{\min}\sum_{j=k+1}^{n-1}\gamma_{j}}$ if $k \neq n-1$, where $\lambda_{\min}$ is the smallest eigenvalue of $\Gamma_{m}$. With calculus analogous to the ones for the bound of the martingale term in the proof of Proposition \ref{majexp}, one can check that if $k \leq n/2$,
\begin{align*}
 \nrm{ \beta_{n-1}\beta_{k}^{-1} \gamma_{k}\xi_{k+1}} & \leq 2c_{2}e^{-2\lambda_{\min}c_{\gamma}n^{1-\alpha}}\gamma_{1}. 
\end{align*}
Moreover, if $k \geq n/2$ and $k \neq n-1$,
\begin{align*}
2c_{2}\gamma_{k}e^{-\lambda_{\min}\sum_{j=k+1}^{n-1}\gamma_{j}} & \leq 2c_{2}\gamma_{k} \leq 2c_{2}2^{\alpha}c_{\gamma}\frac{1}{n^{\alpha}}.
\end{align*}
Finally, if $k=n-1$,
\begin{align*}
\nrm{\beta_{n-1}\beta_{k}^{-1} \gamma_{n-1}\xi_{n}} & \leq c_{\gamma}2^{\alpha}\frac{1}{n^{\alpha}}.
\end{align*}
Let $C_{N}:= \max \left\lbrace \sup_{n\geq 1}\left\lbrace e^{-2\lambda_{\min}c_{\gamma}n^{1-\alpha}}n^{\alpha}\right\rbrace , 2c_{2} , 1 \right\rbrace$, thus for all $n \geq 1$, \\
$\sup_{k \leq n-1}  \left\lbrace \| \beta_{n-1}\beta_{k}^{1}\gamma_{k}\xi_{k+1} \| \right\rbrace \leq \frac{C_{N}}{n^{\alpha}}$. So we take
\begin{align*}
N_{n} &=\frac{C_{N}}{n^{\alpha}} .
\end{align*}

\noindent\textit{Step 2: Choice of $\sigma_{n}^{2}$.} \\
In the same way, for $n$  large enough, we have 
\[
\sum_{k=1}^{n-1}\mathbb{E}\left[ \left\| \beta_{n-1}\beta_{k}^{-1}\gamma_{k}\xi_{k+1}\right\|^{2} |\mathcal{F}_{k} \right] \leq \frac{2^{\alpha +1}c_{\gamma}}{c_{m}}\frac{1}{n^{\alpha}}.
\]
 Indeed, we can split the sum into two parts, the first one converges exponentially fast to $0$, and is smaller than the second one from a certain rank. For $n$ large enough, we can take
\begin{equation}\label{choixsigma}
\sigma_{n}^{2}= c_{\gamma}\frac{2^{1+\alpha}}{c_{m}}\frac{1}{n^{\alpha}}.
\end{equation}
Using inequality (\ref{majinterv}) and Corollary \ref{corpinbern},
\begin{align*}
\mathbb{P}\left[ \| Z_{n}-m \| \geq t \right] &\leq 2\exp \left( -\frac{(t/2)^{2}}{2(\sigma_{n}^{2} + N_{n}(t/2)/3)}\right) + \frac{C_{1}e^{-C_{1}'n^{1-\alpha}}}{t^{2}}+\frac{C_{2}}{n^{\alpha}}\frac{1}{t}=:f(t,n) .
\end{align*}
We look for values of $t$ for which $f(t,n) \leq \delta$. We search to solve:
\begin{align*}
2\exp \left( -\frac{(t/2)^{2}}{2(\sigma_{n}^{2} + N_{n}t/6)}\right) & \leq \delta /2 , \\
\frac{C_{1}e^{-C_{1}'n^{1-\alpha}}}{t^{2}} & \leq \delta /4 , \\
\frac{C_{2}}{n^{\alpha}}\frac{1}{t} & \leq \delta /4 .
\end{align*}
We get (see \cite{Tarres} , Appendix A, for the exponential term):
\begin{align*}
t & \geq 4 \left( \frac{N_{n}}{3}+\sigma_{n} \right) \ln \frac{4}{\delta} , \\
t & \geq 2\sqrt{ \frac{C_{1}e^{-C_{1}'n^{1-\alpha}}}{\delta}} , \\
t & \geq 4 \frac{C_{2}}{n^{\alpha}}\frac{1}{\delta}.
\end{align*}
Let us take a rank $n_\delta$ such that for all $n \geq n_{\delta}$, with (\ref{choixsigma}),
\begin{align*}
4 \left( \frac{N_{n}}{3}+\sigma_{n} \right) \ln \frac{4}{\delta} & \geq  2\sqrt{ \frac{C_{1}e^{-C_{1}'n^{1-\alpha}}}{\delta}},\\
4 \left( \frac{N_{n}}{3}+\sigma_{n} \right) \ln \frac{4}{\delta} & \geq 4 \frac{C_{2}}{n^{\alpha}}\frac{1}{\delta}.
\end{align*}
Thus, for all $n \geq n_{\delta}$, with probability at least $1-\delta$:
\begin{align*}
\| Z_{n}-m \| &\leq 4 \left( \frac{N_{n}}{3}+\sigma_{n} \right) \ln \frac{4}{\delta}.
\end{align*}
\end{proof}

\subsection{Proof of Theorem \ref{theointervmoyennise}}

Since $\mathbb{E}\left[ \| Z_{n}-m \|^{2} \right] \leq \frac{C'}{n^{\alpha}}$, applying Cauchy-Schwarz's inequality, we have $\mathbb{E}\left[ \| Z_{n}-m \| \right] \leq \sqrt{\frac{C'}{n^{\alpha}}}$. These bounds are useful to prove that the first terms in equation (\ref{decinterv}) are negligible. Indeed,
\begin{align*}
\mathbb{E}\left[ \left\| \frac{T_{n+1}}{n\gamma_{n}}\right\|^{2} \right] & \leq \frac{n^{2\alpha}}{c_{\gamma}n^{2}}\mathbb{E}\left[ \| Z_{n+1}-m \|^{2} \right] \\
& \leq \frac{n^{2\alpha}}{c_{\gamma}n^{2}} \frac{C'}{(n+1)^{\alpha}} \\
& \leq \frac{2^{\alpha}C'}{c_{\gamma}}\frac{1}{n^{2-\alpha}} .
\end{align*}
Since $\alpha < 1$, we have that $\frac{2-\alpha}{2}> \frac{1}{2}$. Moreover, since $0<\gamma_{k+1}^{-1}-\gamma_{k}^{-1}\leq \alpha c_{\gamma}^{-1}k^{\alpha-1}$, there is a positive constant $C_{1}$ such that:
\begin{align*}
\mathbb{E}\left[ \left\| \frac{1}{n}\sum_{k=2}^{n}T_{k}\left(  \gamma_{k}^{-1}-\gamma_{k+1}^{-1}\right) \right\| \right] & \leq \frac{\alpha c_{\gamma}^{-1}}{n}\sum_{k=2}^{n} \mathbb{E}\left[ \| T_{k} \| \right] k^{\alpha -1} \\
& \leq \frac{\alpha c_{\gamma}^{-1}\sqrt{C'}}{n}\sum_{k=2}^{n/2 -1}k^{\alpha /2 -1} \\
& \leq \frac{C_{1}}{n^{1-\alpha /2}}.
\end{align*}
Note also that since $\alpha < 1$, we have $1-\alpha /2 \geq 1/2$. Moreover, since $\| \delta_{n} \| \leq C_{m} \| T_{n} \|^{2}$, there is a positive constant $C_{2}$ such that
\begin{align*}
\mathbb{E}\left[ \left\| \frac{1}{n}\sum_{k=1}^{n} \delta_{k} \right\| \right] & \leq \frac{C_{m}}{n}\sum_{k=1}^{n}\mathbb{E}\left[ \| T_{k} \|^{2}\right] \\
& \leq \frac{C_{m}C'}{n}\sum_{k=1}^{n}k^{-\alpha} \\
& \leq C_{2}\frac{1}{n^{\alpha}} .
\end{align*}
Finally, there is a positive constant $C_{3}$ such that $\mathbb{E}\left[ \left\| \frac{T_{1}}{\gamma_{1}n}\right\| \right] \leq \frac{C_{3}}{n}$.

We now study the martingale term. Let $M$ be a constant and $\left(\sigma_{n}\right)$ be a sequence of positive real numbers defined by:
\begin{align*}
M & := 2 \geq \sup_{i} \| \xi_{i} \| , \\
\sigma_{n}^{2}&  := n \geq \sum_{k=1}^{n} \mathbb{E} \left[ \| \xi_{k}\|^{2}|\mathcal{F}_{k-1} \right] .
\end{align*}
Applying Pinelis-Bernstein's Lemma, we have for all $t>0$, 
\begin{align*}
\mathbb{P}\left( \sup_{1 \leq k \leq n} \left\| \widehat{M}_{k+1} \right\| \geq t \right) &\leq 2\exp \left[ -\frac{t^{2}}{2\left(\sigma_{n}^{2}+Mt/3 \right)}\right]. 
\end{align*}
Consequently, 
\begin{align*}
\mathbb{P}\left( \frac{\left\| \widehat{M}_{n+1}\right\|}{n}\geq t \right) & \leq \mathbb{P}\left( \sup_{1\leq k \leq n}\left\| \widehat{M}_{k+1} \right\| \geq tn \right) \\
& \leq 2\exp \left[ -\frac{t^{2}n^{2}}{2\left( \sigma_{n}^{2} + Mtn/3 \right)}\right] \\
& \leq 2 \exp \left[ -\frac{t^{2}}{2\left( \sigma_{n}^{2}/n^{2}+Mt/3n \right)}\right] \\
& \leq 2\exp \left[ -\frac{t^{2}}{2\left( \sigma_{n}'^{2}+N_{n}'t/3 \right)}\right] ,
\end{align*}
with $\sigma_{n}'^{2} := n^{-1}$ and $N_{n}':= 2 n^{-1}$. As in the proof of Theorem \ref{interv}, there are three positive constants $C_{1}'$, $C_{2}'$ and $C_{3}'$ such that for all $t>0$,
\[
 \mathbb{P}\left[ \left\| \Gamma_{m}\left( \overline{Z}_{n}-m \right) \right\| \geq t \right]  \leq 2\exp \left[ -\frac{(t/2)^{2}}{2\left( \sigma_{n}'^{2}+N_{n}'t/6 \right)}\right] + \frac{C_{1}'}{n^{1-\alpha /2}} + \frac{C_{2}'}{n^{\alpha}} + \frac{C_{3}'}{n} =: g(t,n).
\]
We search values of $t$ such that $g(t,n) \leq \delta$. We have to solve the following system of inequalities,
\begin{align*}
2 \exp \left[ -\frac{(t/2)^{2}}{2(\sigma_{n}'^{2}+N_{n}t/6 )}\right] & \leq \delta /2 \\
\frac{C_{1}'}{tn^{1-\alpha/2}} & \leq \delta /6 \\
\frac{C_{2}'}{tn^{\alpha}} & \leq \delta /6 \\
\frac{C_{3}'}{tn} & \leq \delta /6 .
\end{align*}
We get (see \cite{Tarres} , Appendix A, for the martingale term):
\begin{align*}
t & \geq 4 \left( \frac{N_{n}'}{3}+ \sigma_{n}' \right) \ln \left( \frac{4}{\delta}\right) \\
t & \geq \frac{6C_{1}'}{\delta} \frac{1}{n^{1-\alpha /2}}\\
t & \geq \frac{6C_{2}'}{\delta} \frac{1}{n^{\alpha}} \\
t & \geq \frac{6C_{3}'}{\delta} \frac{1}{n}.
\end{align*}
Since $\left( \frac{N_{n}'}{3}+\sigma_{n}' \right) = \frac{2}{3n}+ \frac{1}{\sqrt{n}}$, the other terms are negligible for $n$ large enough and we can consider a rank $n_\delta$ as in (\ref{def:ndelta}).

\bibliographystyle{apalike}
\bibliography{biblio_redaction}

\end{document}